\documentclass[12pt]{article}

\usepackage[utf8]{inputenc}
\usepackage[]{datetime}
\usepackage[noadjust]{cite}
\usepackage{amssymb}
\usepackage{graphicx}
\usepackage{enumerate}
\usepackage{amsfonts}
\usepackage{amsmath}
\usepackage{stmaryrd}
\usepackage{xspace}
\usepackage{paralist}
\usepackage{hyperref}
\usepackage{amsthm}
\usepackage[table]{xcolor} 
\usepackage{tikz}
\usepackage{ifthen}

\hypersetup{
  pdftitle = {Induced minors and well-quasi-ordering},
  colorlinks = true,
  linkcolor = black!30!red,
  citecolor = black!30!green
}

\newcommand{\remref}[1]{\hyperref[#1]{(R\ref*{#1})}}
\newcommand{\itemref}[1]{\hyperref[#1]{(\ref*{#1})}}

\usetikzlibrary{decorations.pathreplacing}
\usetikzlibrary{decorations.pathmorphing}
\usetikzlibrary{decorations.markings}
\tikzset{black node/.style={draw, circle, fill = black, minimum size = 5pt, inner sep = 0pt}}
\tikzset{white node/.style={draw, circle, fill = white, minimum size = 5pt, inner sep = 0pt}}
\tikzset{normal/.style = {draw=none, fill = none}}

\newtheorem{theorem}{Theorem}
\newtheorem{lemma}{Lemma}
\newtheorem{corollary}{Corollary}
\newtheorem{proposition}{Proposition}

\theoremstyle{remark}
\newtheorem{remark}{Remark}

\theoremstyle{definition}

\newcommand{\N}{\mathbb{N}}
\newcommand{\intv}[2]{\left \llbracket #1, #2 \right \rrbracket}
\newcommand{\card}[1]{\left | #1 \right |} 
\DeclareMathOperator{\powset}{\mathcal{P}}
\DeclareMathOperator{\fpowset}{\mathcal{P}^{<\omega}}

\DeclareMathOperator{\cc}{cc} 
\newcommand{\induced}[2]{{#1[#2]}}

\newcommand{\patht}[3]{{#1 #2 #3}}

\newcommand{\lleq}{\preceq} 
\newcommand{\linm}{\leq_\mathrm{im}} 
\newcommand{\nlinm}{\nleq_\mathrm{im}} 
\newcommand{\lisgr}{\leq_\mathrm{isg}} 
\newcommand{\lisgrp}{\leq_\mathrm{isg'}} 
\newcommand{\lctr}{\leq_\mathrm{c}} 
\newcommand{\subseq}{\mathbin{=^\star}}

\DeclareMathOperator{\excl}{Excl_{im}}

\newcommand{\gem}{\mathrm{Gem}} 
\DeclareMathOperator{\lab}{lab} 
\newcommand{\seqb}[1]{\left \langle #1 \right \rangle} 
\newcommand{\kmult}{\mathcal{K}_{{\N}^\star}} 
\newcommand{\htg}{\widehat{K}_4} 
\DeclareMathOperator{\fst}{fst}
\DeclareMathOperator{\lst}{lst}
\newcommand{\opath}{\mathcal{OP}} 
\newcommand{\wm}{\mathcal{WM}} 
\newcommand{\ci}{\mathcal{CI}} 

\newcommand{\arxiv}[1]{\href{http://arxiv.org/abs/#1}{\tt arXiv:#1}}

\title{Induced minors and well-quasi-ordering\thanks{This work was partially done while J.\ Błasiok was student at the Institute of Computer Science, University of Warsaw, Poland and while J.{-F}.\ Raymond was affiliated to LIRMM, Université de Montpellier, France and to the Institute of Computer Science, University of Warsaw, Poland. The research was supported by the Foundation for Polish Science (Jarosław Błasiok and Marcin Kamiński), the (Polish) National Science Centre grants SONATA UMO-2012/07/D/ST6/02432 (Marcin Kamiński and Jean-Florent Raymond) and PRELUDIUM
    2013/11/N/ST6/02706 (Jean-Florent Raymond), the Warsaw Center of Mathematics and Computer Science (Jean-Florent Raymond and Théophile Trunck), and the European Research Council (ERC) under the European
Union’s Horizon 2020 research and innovation programme, ERC consolidator grant DISTRUCT, agreement No 648527 (Jean-Florent Raymond).
    Emails: \href{mailto:jblasiok@g.harvard.edu}{\texttt{jblasiok@g.harvard.edu}}, \href{mailto:mjk@mimuw.edu.pl}{\texttt{mjk@mimuw.edu.pl}}, \href{mailto:raymond@tu-berlin.de}{\texttt{raymond@tu-berlin.de}}, and  \href{mailto:theophile.trunck@ens-lyon.org}{\texttt{theophile.trunck@ens-lyon.org}}.}}

\author{Jarosław Błasiok%
  \thanks{School of Engineering and Applied Sciences, Harvard University, United States.}, %
  Marcin Kamiński%
  \thanks{Institute of Computer Science, University of Warsaw, Poland.}\\ %
  Jean-Florent Raymond%
  \thanks{Technische Universität Berlin, Germany.}, %
  Théophile Trunck%
  \thanks{LIP, ÉNS de Lyon, France.}
}

\date{}
\begin{document}
\maketitle

\begin{abstract}
  A graph $H$ is an induced minor of a graph $G$ if it can be
  obtained from an induced subgraph of $G$ by contracting
  edges. Otherwise, $G$ is said to be $H$-induced
  minor-free.  Robin Thomas showed that $K_4$-induced minor-free
  graphs are well-quasi-ordered by induced minors [\emph{Graphs
    without $K_4$ and well-quasi-ordering}, Journal of Combinatorial
  Theory, Series B, 38(3):240--247, 1985].
  
  We provide a dichotomy theorem for $H$-induced minor-free graphs and
  show that the class of $H$-induced minor-free graphs is
  well-quasi-ordered by induced minors if and only if $H$
  is an induced minor of the $\gem$ (the path on 4 vertices plus a
  dominating vertex) or of the graph obtained by adding a vertex of
  degree 2 to the complete graph on 4 vertices. To this end we prove two
  decomposition theorems which are of independent interest.

  Similar dichotomy results were previously given for subgraphs by Guoli
  Ding in [\emph{Subgraphs and well-quasi-ordering}, Journal of Graph
  Theory, 16(5):489–502, 1992] and for induced subgraphs by Peter
  Damaschke  in [\emph{Induced subgraphs and well-quasi-ordering},
  Journal of Graph Theory, 14(4):427--435, 1990].
\end{abstract}

\section{Introduction}
\label{sec:intro}

A \emph{well-quasi-order}
(\emph{wqo} for short) is a quasi-order which contains no infinite
decreasing sequence and no infinite collection of pairwise incomparable
elements (called an \emph{antichain}).
 One of the most important
results in this field is arguably the theorem by Robertson and
Seymour which states that graphs are well-quasi-ordered by the minor
relation~\cite{Robertson2004325}. Other natural containment relations are not so generous; they usually do not wqo all graphs. In the last decades, much attention has
been brought to the following question: given a partial order $(S,
\lleq)$, what subclasses of $S$ are well-quasi-ordered by $\lleq$?
For instance, 
Fellows et al.\ proved in~\cite{fellows2009well} that graphs with
bounded feedback-vertex-set are well-quasi-ordered by topological minors.
Another result is that of Oum~\cite{doi:10.1137/050629616} who  proved  that graphs of
bounded rank-width are wqo by vertex-minors. Other papers considering
this question include~\cite{JGT:JGT3190140406, Thomas1985240,
  Ding:1992:SW:152782.152791, 2014arXiv1412.2407K, Lui2014,
  JGT:JGT4, Ding20091123, Atminas14b, Petkovsek2002375, DaRaTh10}.

One way to approach this problem is to consider graph classes defined
by excluded substructures.
In this direction, Damaschke proved in~\cite{JGT:JGT3190140406} that a class of graphs
defined by one forbidden induced subgraph $H$ is wqo by the induced subgraph relation
if and only if $H$ is the path on four vertices. Similarly, a bit later Ding proved in \cite{Ding:1992:SW:152782.152791} an analogous result for the subgraph
relation. Other authors also considered this problem (see for
instance~\cite{Kaminski2016well, Korpelainen20111813,JGT:JGT20528,
  Cherlin:2011:FSC:2645908.2646015}). In this paper, we
provide an answer to the same question for
the induced minor relation, which we denote~$\linm$.
This generalizes a result of Thomas who proved that graphs with no $K_4$-minor are wqo by~$\linm$~\cite{Thomas1985240}.
Before stating our main result, let us
introduce two graphs which play a major role in this paper (see~\autoref{fig:h123}). The first
one, $\htg$, is obtained by adding a vertex of degree two to
$K_4$, and the second one, called the $\gem$, is constructed by adding a
dominating vertex to $P_4$.
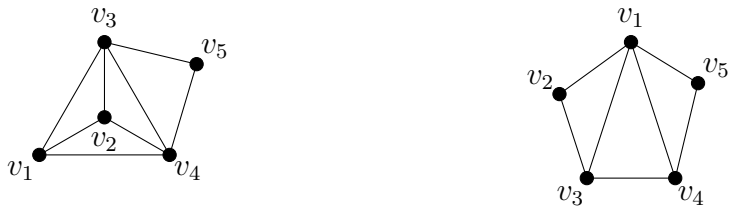
\begin{figure}[ht]
  \centering
  \begin{tikzpicture}[every node/.style = black node, scale = 1]
    \begin{scope}
      \draw (90:1) node[label=$v_3$] {} -- (210:1) node[label=210:$v_1$] {} -- (330:1) node[label=-30:$v_4$] {} -- cycle;
      \draw (90:1) -- (0,0) node[label=-90:$v_2$] {};
      \draw (210:1) -- (0,0);
      \draw (330:1) -- (0,0);
      \draw (90:1) -- (30:1.4142) node[label=30:$v_5$] {} -- (330:1);
    \end{scope}
    \begin{scope}[xshift = 7cm]
      \draw (162:1) node[label = 162:$v_2$] {}
      -- (234:1) node[label = 234:$v_3$] {}
      -- (306:1) node[label = 306:$v_4$] {}
      -- (387:1) node[label = 387:$v_5$] {}
      -- (90:1) node[label = 90:$v_1$] {} -- cycle;
      \draw (90:1) -- (234:1);
      \draw (90:1) -- (306:1);
    \end{scope}
  \end{tikzpicture}
  \caption{The graph~$\htg$ (on the left) and the  $\gem$
    (on the right).}
  \label{fig:h123}
\end{figure}

\section{Induced minors and well-quasi-ordering}

Our main result is the following.

\begin{theorem}[Dichotomy Theorem]\label{t:dich}
  Let $H$ be a graph.  The class of $H$-induced minor-free graphs is wqo by
  $\linm$ iff $H$ is an induced minor of $\htg$ or the $\gem.$
\end{theorem}

Our proof naturally has two parts: for different graphs $H$, we
need to show wqo of $H$-induced minor-free graphs or exhibit an
$H$-induced minor-free antichain.

\paragraph{Classes that are wqo.}

The following two theorems describe the structure of graphs with $H$
forbidden as an induced minor, when $H$ is  $\htg$ and the $\gem$,
respectively.

\begin{theorem}[Decomposition of $\htg$-induced minor-free graphs]\label{t:dec-h3}
  Let $G$ be a 2-connected graph such that $\htg \not \linm G$. Then one of the following holds:
  \begin{enumerate}[(i)]
  \item $K_4 \not \linm G$; or\label{h3-1}
  \item $G$ is a subdivision of a graph among $K_4$, $K_{3,3}$, and
    the prism; or\label{h3-2}
  \item $V(G)$ has a partition $(W,M)$ such that $G[W]$ is a wheel
    on at most 5 vertices and $G[M]$ is a complete multipartite graph; or\label{h3-3}
  \item $V(G)$ has a partition $(C,I)$ such that
    $\induced{G}{C}$ is a cycle, $I$ is an independent set and every vertex of $I$ has the same neighborhood on $C$.\label{h3-4}
  \end{enumerate}
\end{theorem}

\begin{theorem}[Decomposition of $\gem$-induced minor-free graph]\label{t:decgem}
  Let $G$ be a 2-connected graph such that $\gem \not \linm G$. Then $G$ has a
  subset $X \subseteq V(G)$ of at most six vertices such that
  every connected component of $G \setminus X$ is either a cograph or
  a path whose internal vertices are of degree two in~$G$.
\end{theorem}

Using the two above structural results, we are able to show the
well-quasi-ordering of the two classes with respect to induced
minors. For every graph $H$, a graph not containing $H$ as induced
minor is said to be \emph{$H$-induced minor-free}.

\begin{theorem}\label{t:exclh3-wqo}
  The class of $\htg$-induced minor-free graphs is wqo by~$\linm$.
\end{theorem}

\begin{theorem}\label{t:exclgem-wqo}
  The class of $\gem$-induced minor-free graphs is wqo by~$\linm$.
\end{theorem}

\paragraph{Organization of the paper.}
After a preliminary section introducing notions and notation used in
this paper, we present in \autoref{sec:antichains} several
infinite antichains for induced minors. \autoref{sec:dicho} is
devoted to the proof of \autoref{t:dich}, assuming
\autoref{t:exclh3-wqo} and~\autoref{t:exclgem-wqo}, the proof of
which are respectively given in \autoref{sec:wqohtg}
and~\autoref{sec:wqogem}. Finally, we give in \autoref{sec:fin} some
directions for further research.

\section{Preliminaries}
\label{s:prelim}

The notation $\intv{i}{j}$ stands for the interval of integers
$\{i, \dots, j\}$. We denote by $\powset(S)$ the power set of a
set~$S$ and by $\fpowset(S)$ the set of all its finite subsets.

\subsection{Graphs and classes}
\label{s:gc}

The graphs in this paper are simple and loopless. Given a graph $G,$ $V(G)$ denotes its vertex set and
$E(G)$ its edge set. For every positive integer $n$, $K_n$ is the complete graph on $n$
vertices and $P_n$ is the path on $n$ vertices. For every integer $n\geq 3$, $C_n$ is the cycle on $n$ vertices. For $H$ and $G$
graphs, we write $H + G$ the disjoint union of $H$ and $G$. The \emph{complement} of a graph $G$, denoted by $\overline{G}$, is obtained by remplacing every edge by a non-edge, and vice-versa. Also, for
every $k \in \N$, $k \cdot G$ is the disjoint union of
$k$ copies of $G$.  For every pair $u,v$ of vertices a path $P$ there
is exactly one subpath in $P$ between $u$ and $v$, that we denote by
$\patht{u}{P}{v}$. Two vertices $u,v \in V(G)$ are said to be
\emph{adjacent} if $\{u,v\} \in E(G).$ The \emph{neighborhood} of
$v \in V(G),$ denoted $N_G(v),$ is the set of vertices
that are adjacent to $v.$ If $H$ is a subgraph of $G$, we write
$N_H(v)$ for $N_G(v) \cap V(H).$ Given two sets
$X,Y$ of vertices of a graph, we say that there is an edge between $X$
and $Y$ (or that $X$ and $Y$ are adjacent) if there is $x \in X$ and
$y \in Y$ such that $\{x,y\} \in E(G).$ The number of connected
components of a graph $G$ is denoted $\cc(G).$ We call \emph{prism}
the complement of~$C_6$.

A \emph{cograph} is a graph not containing the path on four vertices
as induced subgraph.
The following notion will be used when decomposing graphs not
containing $\gem$ as induced minor.
An induced subgraph of a graph $G$ is said to be \emph{basic in $G$}
if it is either a cograph, or an induced path whose internal vertices
are of degree two in~$G$.
A \emph{linear forest} is a disjoint union of paths.
The \emph{closure} of a class $\mathcal{G}$ by a given operation is
the class obtained from graphs of $\mathcal{G}$ by a finite
application of this operation.

\paragraph{Complete multipartite graphs.} A graph $G$ is said to be \emph{complete multipartite} if its vertex
set can be partitioned into sets $V_1, \dots, V_k$ (for some positive
integer $k$) in a way such that two vertices of $G$ are adjacent iff they belong to
different~$V_i$'s. The class of complete multipartite graphs is
referred to as~$\kmult.$

\paragraph{Wheels.}
  For every positive integer $k$, a \emph{$k$-wheel} is a graph obtained from $C +
  K_1,$ where $C$ is an induced cycle of order at least $k,$ by connecting the
  isolated vertex to $k$ distinct vertices of the cycle. $C$ is said
  to be the \emph{cycle} of the $k$-wheel, whereas the vertex
  corresponding to $K_1$ is its \emph{center}.

\paragraph{Cutsets.}
In a graph $G$, a $K_2$-cutset (resp. $\overline{K_2}$-cutset) is a subset $S
\subseteq V(G)$ such that $G-S$ is not connected and
$\induced{G}{S}$ is isomorphic to $K_2$ (resp.\ $\overline{K_2}$).

\paragraph{Labels and roots.}
  Let $(\Sigma, \lleq)$ be a poset. A $(\Sigma, \lleq)$-labeled graph
  is a pair $(G, \lambda)$ such that $G$ is a graph, and
  $\lambda \colon V(G) \to \fpowset(\Sigma)$ is a function referred as the
  \emph{labeling of the graph}. For the sake of simplicity, we will refer to the labeled graph
  of a pair $(G, \lambda)$ by $G$ and to $\lambda$ by $\lambda_G$. If $\mathcal{G}$
  is a class of (unlabeled) graphs, $\lab_{(\Sigma, \lleq)}(\mathcal{G})$ denotes the
  class of $(\Sigma, \lleq)$-labeled graphs of $\mathcal{G}.$ Observe that any
  unlabeled graph can be seen as a $\emptyset$-labeled~graph.
  A \emph{rooted graph} is a graph with a distinguished edge called \emph{root}.

  Labels will allow us to focus on labeled 2-connected graphs, as stated in the following proposition.
  \begin{proposition}[\cite{fellows2009well}]\label{p:2c-labels}
    Let $\mathcal{G}$ be a class of graphs that is closed by taking induced subgraphs. If for any wqo $(S, \lleq)$
    the class of $(S,\lleq)$-labeled 2-connected graphs of $\mathcal{G}$
    is wqo by $\linm$, then $\mathcal{G}$ is wqo by $\linm$.
  \end{proposition}
  
\subsection{Sequences, posets and well-quasi-orders}

In this section, we introduce basic definitions and facts related to
the theory of well-quasi-orders. In particular, we recall that
being well-quasi-ordered is preserved by several operations.
 
\paragraph{Sequences.} A \emph{sequence} of elements of a set~$A$ is an ordered countable
collection of elements of~$A.$ Unless otherwise stated, sequences are
finite. The sequence of elements $s_1, \dots, s_k \in A$ in this
order is denoted by $\seqb{s_1, \dots, s_k}.$ 
We~use the notation $A^\star$ for the
class of all finite sequences over $A$ (including the empty
sequence). The length of a finite
sequence $s\in A^\star$ is denoted~by~$|s|$. 

\paragraph{Posets ans wqos.}
A \emph{partially ordered set} (\emph{poset} for short) is a pair $(A,
\lleq)$ where $A$ is a set and $\lleq$ is a binary relation on $A$
which is reflexive, antisymmetric and transitive.
An \emph{antichain} is a sequence of pairwise non-comparable elements.
In a sequence $\seqb{x_i}_{i\in I \subseteq \N}$ of a poset $(A, \lleq),$ a pair
$(x_i,x_j),$ $i,j \in I$ is a \emph{good pair} if $x_i \lleq x_j$
and~$i<j.$ A poset $(A, \lleq)$ is a \emph{well-quasi-order}
(\emph{wqo} for short)\footnote{Usually in literature the term 
	well-quasi-order is defined for more general
	structures than posets, namely {\it quasi-orders}. Those relations 
	are like posets, except they are not required to be antisymmetric.
	This is mere technical detail, as
	every poset is a quasi-order, and from a quasi-order one can make a
	poset by taking a quotient by the equivalence relation 
	$a \lleq b \land b \lleq a$.},
and its elements are said to be \emph{well-quasi-ordered} (\emph{wqo} for short) by $\lleq,$
if every infinite sequence has a good pair, or
equivalently, if $(A, \lleq)$ has neither an infinite decreasing sequence, nor
an infinite antichain. An infinite sequence containing no good pair is
called an \emph{bad sequence}.

\paragraph{Ordering sequences.} For any partial order $(A, \lleq),$ we define
the relation $\lleq^\star$ on $A^\star$ as follows: for every $r
=\seqb{r_1,\dots, r_p}$ and $s = \seqb{s_1,\dots, s_q}$ of
$A^\star,$ we have $r \lleq^\star s$ if there is a increasing
function $\varphi \colon \intv{1}{p} \to \intv{1}{q}$ such that for
every $i \in \intv{1}{p}$ we have~$r_i \lleq s_{\varphi(i)}.$ Observe that
$\subseq$ is then the subsequence relation.
This order relation is extended to the class $\fpowset(A)$ of finite
subsets of $A$ as follows, generalizing the subset relation: for every
$B,C \in \fpowset(A)$, we write $B \lleq^{\powset{}} C$ if there is an
injection $\varphi \colon B \to C$ such that $\forall x \in B,\ x
\lleq \varphi(x).$ Observe that $=^{\powset{}}$ is the subset relation.

\paragraph{Monotonicity.}
In order to stress that the domain $(A, \lleq_A)$ and codomain $(B, \lleq_B)$ of a function $\varphi$ are
posets, we sometimes use the notation $\varphi \colon (A, \lleq_A) \to (B, \lleq_B)$.
A function $\varphi \colon (A, \lleq_A) \to (B, \lleq_B)$ is said to
be \emph{monotone} if it satisfies the following property:
\[
\forall x,y \in A,\ x \lleq_A y \Rightarrow f(x) \lleq_B f(y).
\]

In order to prove that a function is monotone, one can focus on each argument separately, as noted in the following remark.
\begin{remark}\label{r:mono-pair}
  Let $(A, \lleq_A),$ $(B, \lleq_B),$ and $(C, \lleq_C)$ be
  posets and let {$f \colon (A \times B, \lleq_A \times \lleq_B) \to
  (C, \lleq_C)$} be a function. If we have both
  \begin{align}
    \forall a \in A,\ \forall b,b' \in B,\ b \lleq_B b' &\Rightarrow
    f(a,b) \lleq_C f(a,b')\label{eq:mono1}\\
    \text{and}\ \forall a,a' \in A,\ \forall b \in B,\ a \lleq_A a' &\Rightarrow
    f(a,b) \lleq_C f(a',b)\label{eq:mono2}
  \end{align}
  then $f$ is monotone.
  Indeed, let $(a,b), (a',b') \in A \times B$ be such that $(a,b)
  \lleq_A \times \lleq_B (a',b').$ By definition of the relation $\lleq_A
  \times \lleq_B,$ we have both $a \lleq a'$ and $b \lleq b'.$
  From line \itemref{eq:mono1} we get that $f(a,b) \lleq_C f(a,b')$ and from
  line \itemref{eq:mono2} that $f(a,b') \lleq_C f(a',b'),$ hence
  $f(a,b)\lleq_C f(a',b')$ by transitivity of $\lleq_C.$ Thus $f$ is
  monotone. Observe that this remark can be generalized to functions with more than
  two arguments.
\end{remark}

Well-quasi-orders can be constructed from smaller ones by simple operations. The following proposition lists well-known properties of well-quasi-orders. For a reference, the reader can refer to \cite{Higman01011952} and \cite{Kruskal1972297}.
\begin{proposition}\label{wqomagic}
  Let $(A, \lleq_A)$ and $(B, \lleq_B)$ be two wqos, then
\begin{itemize}
\item their union $(A \cup B, \lleq_A \cup \lleq_B)$, which is the poset
  defined as~follows:
  \[
  \forall x,y \in A \cup B,\ x \mathbin{\lleq_A \cup \lleq_B} y\ \text{if}\ (x,y \in A\ \text{and}\ x \lleq_A y)\
  \text{or}\ (x,y \in B\ \text{and}\ x \lleq_B y),
\]
is a wqo;
\item their Cartesian product $(A \times B, \lleq_A \times \lleq_B)$ which is the poset defined~by:
  \[
  \forall (a,b),(a',b') \in A \times B,\  (a,b) \lleq_A \times \lleq_B
  (a',b')\ \text{if}\ a \lleq_A a'\ \text{and}\ b \lleq_B b',
\]
is a wqo;
\item $(A^\star, \lleq^\star)$ is a wqo (Higman's Lemma);
\item $(\fpowset(A), \lleq^{\powset{}})$ is a wqo;
\item if $(C, \lleq_C)$ is a poset included in the image of a monotone function with domain $(A, \lleq_A)$, then $(C, \lleq_C)$ is a wqo.
\end{itemize}
\end{proposition}

\subsection{Graph operations and containment relations}
\label{sec:cont}

Most of the common order relations on graphs, sometimes called
\emph{containment relations}, can be defined in two equivalent ways:
either in terms of \emph{graph operations}, or by using \emph{models}.
Let us look closer at them.

\paragraph{Local operations.} If $\{u,v\} \in E(G),$ the
\emph{edge contraction} of $\{u, v\}$ adds a new vertex $w$ adjacent to the
neighbors of $u$ and $v$ and then deletes $u$ and $v$. In the case
where $G$ is labeled, we set $\lambda_G(w) = \lambda_G(u) \cup
\lambda_G(v)$. On the other hand, an \emph{edge subdivision} of $\{u,v\}$ adds
a new vertex adjacent to $u$ and $v$ and deletes the edge~$\{u,v\}.$
The \emph{identification} of two vertices $u$ and $v$ adds the
edge $\{u,v\}$ if it was not already existing, and contracts it.
If $G$ is $(\Sigma, \lleq)$-labeled (for some poset $(\Sigma,
\lleq)$), a \emph{label contraction} is the operation of relabeling a
vertex $v \in V(G)$ with a label~$l$ such that~$l \lleq^{\powset{}} \lambda_G(v)$. 
The motivation for this definition of label contraction is the following. Most of
the time, labels will be used to encode connected graphs into
2-connected graphs. Given a connected graph which is not 2-connected,
we can pick an arbitrary block (i.e.~a maximal 2-connected
component), delete the rest of the graph and label each vertex $v$ by the
subgraph it was attached to in the original graph if $v$ was a cutvertex, and by
$\emptyset$ otherwise. That way, contracting the label of a vertex $v$ in the labeled
2-connected graph corresponds to reducing (for some containment
relation) the subgraph which was dangling at vertex $v$ in the
original graph (see also \autoref{p:2c-labels}).

\paragraph{Models.}
Let $(\Sigma, \lleq)$ be any poset. A \emph{containment model} of a
$(\Sigma, \lleq)$-labeled graph $H$ in a $(\Sigma, \lleq)$-labeled
graph $G$ (\emph{$H$-model} for short) is a function $\mu \colon V(H) \to
\fpowset(V(G))$ satisfying the following conditions:
\begin{enumerate}[(i)]
\item for every two distinct $u,v \in V(H),$ the sets $\mu(u)$
and $\mu(v)$ are disjoint;
\item for every $u \in V(H),$ the subgraph of $G$ induced by
  $\mu(u)$ is connected; 
\item for every $u \in V(H),\ \lambda_H(u) \lleq^\star
  \bigcup_{v \in \mu(u)}\lambda_G(v)$ (label conservation).
\end{enumerate}

When in addition $\mu$ is such that for every two distinct $u,v \in
V(H),$ the sets $\mu(u)$ and $\mu(v)$ are adjacent in $G$ iff
$\{u,v\} \in E(H)$, then $\mu$ is said to be an \emph{induced
  minor model} of $H$ in $G$.

If $\mu$ is an induced minor model of $H$ in $G$ satisfying the
following condition:
\[
\bigcup_{v\in V(H)}\mu(v) = V(G),
\]
then $\mu$ is a \emph{contraction model} of $H$ in~$G$.

If $\mu$ is an induced minor model of $H$ in $G$ satisfying the
following condition:
\[
\forall v \in V(H),\ |\mu(v)| = 1,
\]
then $\mu$ is an \emph{induced subgraph model} of $H$ in~$G$.

An $H$-model in a graph $G$ witnesses the presence of $H$ as
substructure of $G$ (which can be induced subgraph, induced minor,
contraction, etc.), and the subsets of $V(G)$ given by the
image of the model indicate which subgraphs to keep and to contract in
$G$ in order to obtain~$H$.

When dealing with rooted graphs, the aforementioned models must in
addition preserve the root, that is, if $\{u,v\}$ is the root of $H$
then the root of $G$ must have one endpoint in $\mu(u)$ and the other
in $\mu(v)$.

\paragraph{Containment relations.} Local operations and
models can be used to express that a graph is \emph{contained} in an
other one, for various definitions of ``contained''. %
We say that a graph $H$ is an \emph{induced minor} (resp.\ a
contraction, induced subgraph) of a graph $G$ if there is an induced minor model
(resp. a contraction model, an induced subgraph model) $\mu$ of $H$ in
$G$, what we note $H \linm G$ (resp.\ $H \lctr G$, $H \lisgr G$).

Otherwise, $G$ is said to be \emph{$H$-induced minor-free} (resp.\
\emph{$H$-contraction-free}, \emph{$H$-induced subgraph-free}). The
class of $H$-induced minor-free graphs will be referred to as~$\excl(H).$

\begin{remark}
  In terms of local operation, these containment relations are defined
  as follows for every $H,G$ graphs:
  \begin{itemize}
  \item $H \lisgr G$ iff there is a (possibly empty) sequence
    of vertex deletions and label contractions
    transforming $G$ into~$H$;
  \item $H \linm G$ iff there is a (possibly empty) sequence
    of vertex deletions, edge contractions and label contractions
    transforming $G$ into~$H$;
  \item $H \lctr G$ iff there is a (possibly empty) sequence
    of edge contractions and label contractions transforming $G$
    into~$H.$
  \end{itemize}  
\end{remark}

\paragraph{Subdivisions.} A subdivision of a graph $H$, or
$H$-subdivision, is a graph obtained from $H$ by edge subdivisions. The
vertices added during this process are called \emph{subdivision~vertices}.

\paragraph{Containing $K_4$-subdivisions.}
A graph $G$ contains $K_4$ as an induced minor if and only if $G$ contains $K_4$-subdivision as a subgraph. This equivalence is highly specific to the graph $K_4$ and in general neither implication would be true. We will freely change between those two notions for containing $K_4$, depending on which one is more convenient in the given context.

A graph $G$ will be said to contain a \emph{proper $K_4$ subdivision}, if there is some vertex $v \in V(G)$, such that $G \setminus v$ contains a $K_4$-subdivision.

\section{Antichains for induced minors}
\label{sec:antichains}

An infinite antichain is an obstruction for a quasi-order to be a
wqo.
As we will see in \autoref{sec:dicho}, the study of infinite
antichains can provide helpful information when looking for graphs $H$
such that $(\excl(H), \linm)$ is a wqo. In this section we present enumerate some of the
known infinite antichains for induced minors.

In 1985, Thomas~\cite{Thomas1985240} presented an infinite sequence of
planar graphs (also mentioned later in~\cite{roberston1993graph}),
that is an antichain for induced minors. This shows that induced minors
do not well-quasi-order planar~graphs.
The elements of this antichain, called \emph{alternating double
  wheels}, are constructed from an even cycle by adding two
nonadjacent vertices and connecting one to one color class of the
cycle, and connecting the other vertex to the other color class
(cf.~\autoref{fig:thom} for the three first such graphs). This
infinite antichain shows that $(\excl(K_5), \linm)$ is not a wqo since
no alternating double wheel contains $K_5$ as (induced)~minor. As a
consequence, $(\excl(H), \linm)$ is not a wqo as soon as $H$ contains
$K_5$ as induced minor.

Therefore, in the quest for all graphs $H$ such that $(\excl(H),
\linm)$ is wqo, we can focus the cases where $H$ is $K_5$-induced minor-free.

\begin{figure}[ht]
  \centering
  \begin{tikzpicture}[every node/.style = black node, scale = 0.75]
    \clip (-10.5, 1.75) rectangle (3.5,-1.75);
    \begin{scope}[xshift = -8cm] 
   \draw (0,0) circle (1cm);
   \draw (0,0) node[fill = white] (r) {};
   \draw (180:2) node (s) {};
   \foreach \t in {90,270} {
     \draw (\t - 90:1) node[white node] {} to (r);
     \draw (\t:1) node {} to[in = 90 - \t/2, out = 90 + \t/2,
     looseness = 1.5] (s);
     \draw (1, -1) node[normal] {\Large,};
   }
   \end{scope}
   \begin{scope}
   \draw (0,0) circle (1cm);
   \draw (0,0) node[fill = white] (r) {};
   \draw (180:2) node (s) {};
   \foreach \t in {45,135,...,330} {
     \draw (\t - 45:1) node[white node] {} to (r);
     \draw (\t:1) node {} to[in = 90 - \t/2, out = 90 + \t/2,
     looseness = 1.5] (s);
     \draw (1, -1) node[normal] {\Large,};
   }
   \end{scope}
    \begin{scope}[xshift = -4cm] 
   \draw (0,0) circle (1cm);
   \draw (0,0) node[fill = white] (r) {};
   \draw (180:2) node (s) {};
   \foreach \t in {60,180,300} {
     \draw (\t - 60:1) node[fill = white] {} to (r);
     \draw (\t:1) node {} to[in = 90 - \t/2, out = 90 + \t/2,
     looseness = 1.5] (s);
   }
   \draw (1, -1) node[normal] {\Large,};
   \end{scope}
   \draw (2.5, -1) node[normal] {\Large $\dots$};
  \end{tikzpicture}
  \caption{Thomas' alternating double wheels.}
  \label{fig:thom}
\end{figure}
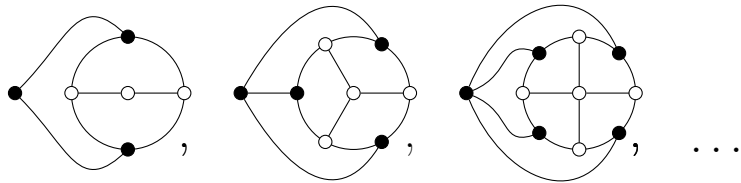

The infinite antichain $\mathcal{A}_M$ depicted in \autoref{fig:matou} was introduced
in~\cite{ matouvsek1988polynomial}, where it is also proved that none
of its members contains $K_5^-$ as induced minor. Similarly as the above
remark, it follows that if $(\excl(H), \linm)$ is a wqo then $K_5^- \not
\linm H$. Notice that graphs in this antichain have bounded
maximum degree.

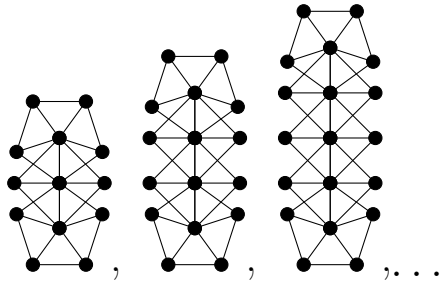
\begin{figure}[h!]
  \centering
  \begin{tikzpicture}[every node/.style = black node, scale = 0.6]
    \clip (1.75, 4) rectangle (11.75,-2.25);
    \foreach \v in {1,...,3} {
      \begin{scope}[xshift = 3*\v cm, yshift = \v cm]
      \begin{scope}
        \draw (-90:1) node {}
        -- (-18:1) node {}
        -- (54:1) node {}
        -- (126:1) node {}
        -- (198:1) node {} 
        -- cycle;
        \draw (0,0) node {}
        edge (-90:1)
        edge (-18:1)
        edge (54:1)
        edge (126:1)
        edge (198:1);
      \end{scope}
      \foreach \i in {1,..., \v}{
        \begin{scope}[yshift = -\i cm]
          \draw (0:1) node {}
          -- (90:1) node {}
          -- (180:1) node {}
          -- (-90:1) node {}
          -- cycle;
          \draw (0,0) node {}
          edge (90:1)
          edge (180:1)
          edge (-90:1)
          edge (0:1);
        \end{scope}
      }
      \pgfmathparse{int(-\v -1)} \let \vp \pgfmathresult
      \begin{scope}[yshift = \vp cm, rotate = 180]
        \draw (-90:1) node {}
        -- (-18:1) node {}
        -- (54:1) node {}
        -- (126:1) node {}
        -- (198:1) node {} 
        -- cycle;
        \draw (0,0) node {}
        edge (-90:1)
        edge (-18:1)
        edge (54:1)
        edge (126:1)
        edge (198:1);
      \end{scope}
      \draw[yshift = -\v cm] (1.25, -2) node[normal] {\Large,};
    \end{scope}
    }
    \draw (11, -2) node[normal] {\Large$\dots$};
  \end{tikzpicture}
  \caption{The infinite antichain $\mathcal{A}_M$ of Matoušek, Nešetřil, and Thomas.}
  \label{fig:matou}
\end{figure}

An \emph{interval graph} is the intersection graph of segments
of~$\mathbb{R}$. A well-known property of interval graphs that we will
use later is that they do not contains $C_4$ as induced minor.
In order to show that interval graphs are not wqo by $\linm$, Ding
introduced in~\cite{JGT:JGT4} an infinite sequence of graphs defined
as follows.
For every $n\in \N$, $n>2$, let $T_n$ be the set of closed intervals
\begin{itemize}
\item $[i,i]$ for $i$ in $\intv{-2n}{-1} \cup \intv{1}{2n}$;
\item $[-2, 2]$, $[-4, 1]$, $[-2n+1, 2n]$, $[-2n+1, 2n-1]$;
\item $[-2i+1, 2i+1]$ for $i$ in $\intv{1}{n-2}$; and
\item $[-2i, 2i - 2]$ for $i$ in $\intv{3}{n}$.
\end{itemize}
\autoref{fig:dingz} depicts the intervals of $T_6$:
the real axis (solid line) is folded up and an interval $[a,b]$ is
represented by a dashed line between $a$ and~$b$.

\begin{figure}[ht]
  \centering
  \begin{tikzpicture}[every node/.style = normal, scale = 0.75]
    \foreach \i in {1,2,...,12}{
      \draw (-\i, 1) node[label = 90:{\small $\i$}] (N-\i) {};
      \path[draw, dashed] (N-\i.center) .. controls +(-1.5, 1.25) and +(1.5, 1.25) .. (N-\i.center);
      \draw (-\i, -1) node[label = -90:{\small $-\i$}] (M-\i) {};
      \path[draw, dashed] (M-\i.center) .. controls +(-1.5, -1.25) and +(1.5, -1.25) .. (M-\i.center);
    }
    \draw (0,0) node[label = 0:0] (N0) {};
    \node[label = 90:$+\infty$] (N-i) at (-13.5, 1) {};
    \node[label = -90:$-\infty$] (M-i) at (-13.5, -1) {};
    %
    \draw (N-i.center) -- (N-1.center);
    \path[draw] (N-1.center) .. controls +(0.5, 0) and +(0, 0.5) .. (N0.center);
    \path[draw] (N0.center) .. controls +(0, -0.5) and +(0.5, 0) .. (M-1.center);
    \draw (M-1.center) -- (M-i.center);
    %
  \begin{scope}[dashed]
    \foreach \i in {-10, -8, -6, -4}{
      \pgfmathparse{int(\i - 2)}\let \j \pgfmathresult
      \draw (N\i.center) -- (M\j.center);
    }
    \foreach \i in {-9, -7, -5, -3}{
      \pgfmathparse{int(\i + 2)}\let \j \pgfmathresult
    \draw (N\i.center) -- (M\j.center);
  }

  \draw (N-11.center) -- (M-11.center);
  \draw (N-2.center) -- (M-2.center);
  \draw (N-12.center) -- (M-9.center);
  \draw (N-1.center) -- (M-4.center);
  \end{scope}
  \end{tikzpicture}
  \caption{An illustration of the intervals in $T_6$.}
  \label{fig:dingz}
\end{figure}
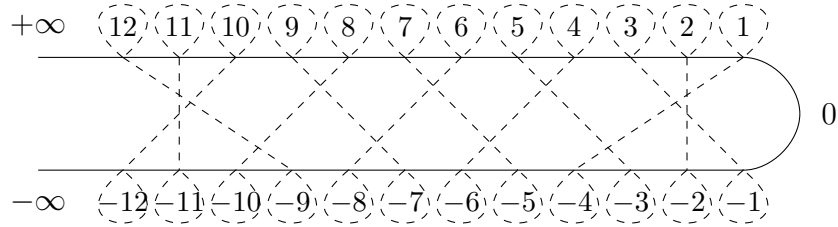

For every $n \in \N$, $n>2$, let $A_n^D$ be the intersection graph of
segments of~$T_n$. Let $\mathcal{A}_D = \seqb{A_n^D}_{n>2}$. Ding
proved in~\cite{JGT:JGT4} that $\mathcal{A}_D$ is an antichain for
$\linm$, thus showing that interval graphs are not wqo by induced~minors.

Let us now present two infinite antichains that were, to our knowledge,
not mentioned elsewhere earlier.
Let $A_{\overline{C}} = \seqb{\overline{C_n}}_{n \geq 3}$ be the
sequence of \emph{antiholes}, whose first elements are
represented in \autoref{fig:antiholes}.

\begin{figure}[ht]
  \centering
  \begin{tikzpicture}[every node/.style = black node, scale = 0.75]
    \foreach \v in {3,...,7}{
      \pgfmathparse{int(3 * (\v - 5))}\let \tmp \pgfmathresult
      \begin{scope}[xshift = \tmp cm]
        \pgfmathparse{360 / \v} \let \angle \pgfmathresult

        \pgfmathparse{\angle - 90}
        \begin{scope}[rotate = \pgfmathresult]
        \foreach \a in {1,...,\v}
        {
          \node[draw] (N\a) at (\a*\angle:1) {};
        }
        {
        \pgfmathparse{\v-2}\let\vmo\pgfmathresult
         \foreach[evaluate={\apo = int(\a+2);}] \a in {1,...,\vmo}{
           \foreach \ap in {\apo,...,\v}{
             \ifthenelse{\NOT\(\a=1 \AND \ap=\v\)}{
             \draw (N\a) -- (N\ap);}{};
           }
         }
         }
        \end{scope}
        \draw (-90:1.75) node[normal] {$\overline{C_\v}$};
        \draw (1,-1) node[normal] {\Large ,};
      \end{scope}
    }
    \draw (8,-1) node[normal] {\Large$\dots$};
  \end{tikzpicture}
  \caption{Antiholes antichain.}
  \label{fig:antiholes}
\end{figure}
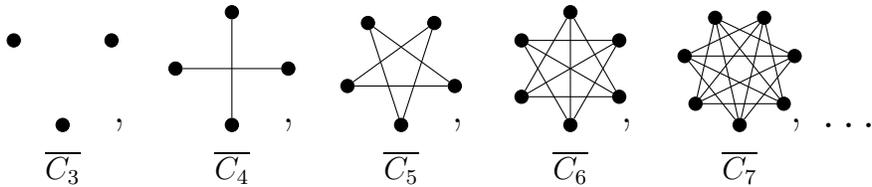

\begin{lemma}\label{l:linfor}
  If $H \linm \overline{C_n}$ and $|V(H)| < n$ for some integer $n \geq 3$, then $\overline{H}$ is a linear
  forest.
\end{lemma}
\begin{proof}
  Towards a contradiction, let us assume that $\overline{H}$ is not a linear forest.

  \smallskip
  \noindent \textit{First case:} $\overline{H}$ has a vertex $v$ of degree at least 3.
  Let $x,y,z$ be three neighbors of $v.$ In the graph
  $\induced{H}{\{v,x,y,z\}},$ the vertex $v$ is adjacent to none of
  $x,y,z.$ In an antihole, every vertex has exactly two non-neighbors,
  so $\induced{H}{\{v,x,y,z\}}$ is not an induced minor of any element of
  $\mathcal{A}_{\overline{C}}.$ In particular, $H \nlinm \overline{C_n}$, a contradiction.

  \smallskip
  \noindent \textit{Second case:} $\overline{H}$ contains an induced cycle as an
  induced subgraph. Then for some integer $n'\geq 3$, we have $\overline{C_{n'}} \linm H \linm \overline{C_n}$.
 Let $\mu$ be an induced minor model of $\overline{C_{n'}}$ in~$\overline{C_{n}}$.
  Let $u,v,w$ be three vertices of $\overline{C_{n'}}$ that appear consecutively and in this order in its complement,~$C_{n'}$. Notice that $v$ is adjacent to none of $u,w$ in $\overline{C_{n'}}$.
  Therefore, there is no edge from $\mu(v)$ to any of $\mu(u)$ and~$\mu(w)$ in $\overline{C_{n}}$. Since a vertex of $\overline{C_{n}}$ has exactly two non-neighbors and distinct vertices have different sets of non-neighbors, we deduce that $\mu(u), \mu(v), \mu(w)$ are singletons and that they contain vertices that are consecutive (in this order) on the cycle of the complement of $\overline{C_{n}}$. Applying this argument for every triple such as $u,v,w$ implies the existence of a cycle of order $n'$ in $C_{n}$, a contradiction.
\end{proof} 

\begin{corollary}
   $\mathcal{A}_{\overline{C}}$ is an antichain.
\end{corollary}

We will meet again the antichain $\mathcal{A}_{\overline{C}}$ in the
proof of \autoref{t:dich}.
Another infinite antichain which shares with $\mathcal{A}_M$ the
properties of planarity and bounded maximum degree is depicted in
\autoref{fig:nestloz}. We will not go more into detail about it
here as this antichain will be of no use in the rest of the paper.

\begin{figure}[ht]
  \centering
  \begin{tikzpicture}[every node/.style = black node, scale = 0.7]
    \begin{scope}[xshift = -2.5cm]
      \draw (-1.75, 0) node {} -- (0,2) node {} -- (1.7, 0) node {} -- (0, -2) node {} -- cycle;
      \draw (0, 2) -- (1, 0) node {};
      \draw (0, -2) -- (-1, 0) node {};
      \draw (-1.75, 0) -- ++(1,2) node {};
      \draw (1.75, 0) -- ++(-1, -2) node {};
      \draw (2, -1) node[normal] {\Large,};
    \end{scope}
    \begin{scope}[xshift = 2cm]
      \draw (-1.75, 0) node {} -- (0,2) node {} -- (1.7, 0) node {} -- (0, -2) node {} -- cycle;
      \draw (-1, 0) node {} -- (0, 2) -- (1, 0) node {} -- (0, -2) -- cycle;
       \draw (-1, 0) -- (0, -1) node {};
      \draw (1, 0) -- (0, 1) node {};
            \draw (-1.75, 0) -- ++(1,2) node {};
      \draw (1.75, 0) -- ++(-1, -2) node {};
      \draw (2, -1) node[normal] {\Large,};
    \end{scope}
    \begin{scope}[xshift = 6.5cm]
      \draw (-1.75, 0) node {} -- (0,2) node {} -- (1.7, 0) node {} -- (0, -2) node {} -- cycle;
      \draw (-1, 0) node {} -- (0, 2) node {} -- (1, 0) node {} -- (0, -2) node (D2) {} -- cycle;
      \draw (-1, 0) -- (0, 1) node {} -- (1,0) -- (0, -1) node {} -- cycle; 
      \draw (0, 1) -- (0.5, 0) node {};
      \draw (0, -1) -- (-0.5, 0) node {};
      \draw (-1.75, 0) -- ++(1,2) node {};
      \draw (1.75, 0) -- ++(-1, -2) node {};
      \draw (2, -1) node[normal] {\Large,};
    \end{scope}
    \begin{scope}[xshift = 11cm]
      \draw (-1.75, 0) node {} -- (0,2) node {} -- (1.7, 0) node {} -- (0, -2) node {} -- cycle;
      \draw (-1, 0) node {} -- (0, 2) node {} -- (1, 0) node {} -- (0, -2) node (D2) {} -- cycle;
      \draw (-1, 0) -- (0, 1) node {} -- (1,0) -- (0, -1) node {} -- cycle; 
      \draw (0, 1) -- (0.5, 0) node (F1) {} -- (0, -1) -- (-0.5, 0) node (F2) {} -- cycle;
      \draw (F1) -- (0, 0.25) node {};
      \draw (F2) -- (0, -0.25) node {};
      \draw (-1.75, 0) -- ++(1,2) node {};
      \draw (1.75, 0) -- ++(-1, -2) node {};
      \draw (2, -1) node[normal] {\Large,};
    \end{scope}
    \draw (15, -1) node[normal] {\Large$\dots$};
  \end{tikzpicture}
  \caption{Nested lozenges.}
  \label{fig:nestloz}
\end{figure}
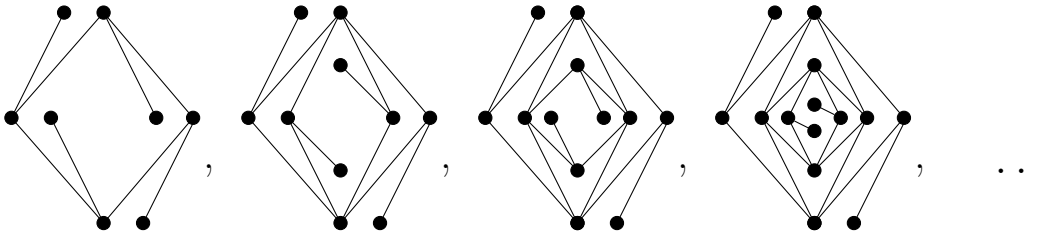

\section{The dichotomy theorem}
\label{sec:dicho}

The purpose of this section is to prove \autoref{t:dich}, that is, to
characterize all graphs $H$ such that $(\excl(H), \linm)$ is a wqo. 
To this end, we will assume \autoref{t:exclh3-wqo} and
\autoref{t:exclgem-wqo}, which we will prove later, in \autoref{sec:wqohtg} and
\autoref{sec:wqogem} respectively.

The main ingredients of the proof are the infinite antichains presented
in \autoref{sec:antichains}, together with \autoref{t:exclh3-wqo} and
\autoref{t:exclgem-wqo}.
Infinite antichains will be used to discard every graph $H$ that is not
induced minor of all but finitely many elements of some infinite antichain. On
the other hand, knowing that $(\excl(H), \linm)$ is a wqo gives that
$(\excl(H'), \linm)$ is a wqo for every $H' \linm H$ in the virtue of the following remark.

\begin{remark}\label{r:exclincl}
  For every $H,H'$ such that $H' \linm H$, we have $\excl(H') \subseteq \excl(H)$.
\end{remark}

As a consequence of \autoref{l:linfor}, if $(\excl(H), \linm)$ is a wqo then $\overline{H}$ is a linear forest.
Because this statement concerns the complement of $H$, we will be led below to work with this
graph rather than with~$H$. The following lemma presents step by step
the properties that we can deduce on $\overline{H}$ by assuming that
$\excl(H)$ is wqo by~$\linm$.

\begin{lemma}\label{l:rkz}
If $(\excl(H), \linm)$ is a wqo, then we have
\begin{enumerate}[(R1)]
\item\label{e:r1} $\overline{H}$ has at most 4 connected components;
\item\label{e:r2} at most one connected component of $\overline{H}$ is not a
  single~vertex;
\item\label{e:r3} the largest connected component of $\overline{H}$ has at most 4 vertices;
\item\label{e:r4} if $n = \card{V(H)}$ and $c =
  \cc(\overline{H})$ then $n \leq 7$ and $\overline{H} = (c - 1) \cdot K_1 +
  P_{n-c+1}$;
\item\label{e:r5} if $\cc(\overline{H}) = 3$ then $\card{V(H)} \leq 5$.
\item\label{e:r6} if $\cc(\overline{H}) = 4$ then $\card{V(H)} \leq 4.$ 
\end{enumerate}
\end{lemma}

\begin{proof}
  \textit{Proof of item~\remref{e:r1}.} The infinite antichain $\mathcal{A}_M$ does not
  contain $K_5$ as (induced) minor, hence $K_5 \not \linm H$ and so
  $\overline{H}$ does not contain $5 \cdot K_1$ as induced
  minor. Therefore it has at  most 4 connected components.

  \smallskip
  \noindent\textit{Proof of items~\remref{e:r2} and~\remref{e:r3}.} The infinite antichain
  $\mathcal{A}_D$ does not contain $C_4$ as induced minor (as it is an
  interval graph), hence neither does $H$. Therefore $\overline{H}$ does not
  contain $2 \cdot P_2$ as induced minor. This implies that $\overline{H}$ does
  not contain $P_5$ as induced minor and that given two connected
  components of $\overline{H}$ at least one must be of order one.
  As connected components of $H$ are paths (by \autoref{l:linfor}),
  the largest connected component of $H$ has order at most~4.

  Item~\remref{e:r4} follows from the above proofs and from
  the fact that $\overline{H}$ is a linear forest.

    \smallskip
  \noindent\textit{Proof of item~\remref{e:r5}.} Similarly as in the proof
  of item~\remref{e:r1}, $\mathcal{A}_M$ does not contain $K_5^-$ as
  induced minor so $\overline{K_5^-} = K_2 + 3 \cdot K_1$ is not an
  induced minor of~$\overline{H}$. If we assume that $\cc(\overline{H}) = 3$
  and $\card{V(H)} \geq 6$ vertices, the largest component of
  $\overline{H}$ is a path on (a least) 4 vertices, so it contains $K_1
  + K_2$ as induced subgraph. Together with the two other (single
  vertex) components, this gives an $K_2 + 3 \cdot K_1$ induced minor,
  a contradiction.

  \smallskip
  \noindent\textit{Proof of item~\remref{e:r6}.} Let us assume that $\cc(\overline{H}) = 4$.
  If the largest connected component has more than one vertex, then
  $\overline{H}$ contains $K_2 + 3 \cdot K_1$ induced minor, which is
  not possible (as in the proof of item~\remref{e:r5}). Therefore
  $\overline{H} = 4 \cdot K_1$ and so $\card{V(\overline{H})}
  = 4$.
\end{proof}

We are now able to describe more precisely graphs $H$ for which
$(\excl(H), \linm)$ could be a~wqo. Let $K_3^+$ be the complement of $P_3 +
K_1$ and let $K_4^-$ be the complement of $K_2 + 2\cdot K_1$, which is
also the graph obtained from $K_4$ by deleting an edge (sometimes
referred as \emph{diamond graph}).

\begin{lemma}\label{l:prethe}
  If $(\excl(H), \linm)$ is a wqo, then $H \linm \htg$ or $H \linm \gem$.
\end{lemma}

\begin{proof}
  Using the information on $\overline{H}$ given by \autoref{l:rkz},
  we can build a table of possible graphs $\overline{H}$ depending on
  $\cc(\overline{H})$ and
  $\card{V(\overline{H})}$. \autoref{t:hbar} is such a
  table: each column corresponds for a number of connected components
  (between one and four according to item~\remref{e:r1}) and each line
  corresponds to an order (at most seven, by item~\remref{e:r4}). A grey
  cell means either that there is no such graph (for instance a graph
  with one vertex and two connected components), or that for all
  graphs $\overline{H}$ matching the number of connected components
  and the order associated with this cell, the poset $(\excl(H),
  \linm)$ is not a wqo.
\begin{table}[ht]
  \centering
  \begin{tabular}{|c|c|c|c|c|}
    \hline
    \phantom{$2^{2^{2^2}}$}$\card{V(H)}\backslash \cc(\overline{H})$& 1 & 2 & 3 & 4\\
    \hline
    1 & $K_1$ & \cellcolor[gray]{0.75}& \cellcolor[gray]{0.75} &\cellcolor[gray]{0.75}\\
    \hline
    2 & $K_2$ & $2\cdot K_1$& \cellcolor[gray]{0.75}&\cellcolor[gray]{0.75}\\
    \hline
    3 & $P_3$ & $K_2 + K_1 $& $3 \cdot K_1$& \cellcolor[gray]{0.75}\\
    \hline
    4 & $P_4$ & $P_3 +  K_1$& $K_2 + 2 \cdot K_1$& $4 \cdot K_1$\\
    \hline
    5 & \cellcolor[gray]{0.75} \remref{e:r3}& $P_4 +  K_1$& $P_3 + 2
    \cdot K_1$& \cellcolor[gray]{0.75} \remref{e:r6}\\
    \hline
    6 & \cellcolor[gray]{0.75} \remref{e:r3}& \cellcolor[gray]{0.75}
    \remref{e:r3}& \cellcolor[gray]{0.75}
    \remref{e:r5}&\cellcolor[gray]{0.75} \remref{e:r6}\\
    \hline
    7 & \cellcolor[gray]{0.75} \remref{e:r3}& \cellcolor[gray]{0.75}
    \remref{e:r3}& \cellcolor[gray]{0.75}
    \remref{e:r5}&\cellcolor[gray]{0.75} \remref{e:r6}\\
    \hline
  \end{tabular}
  \caption{If $(\excl(H), \linm)$ is a wqo, then $\overline{H}$
    belongs to this table.}
  \label{t:hbar}
\end{table}

From \autoref{t:hbar} we can easily deduce \autoref{t:h} of
corresponding graphs $H$.
\begin{table}[ht]
  \centering
  \begin{tabular}{|c|c|c|c|c|}
    \hline
    \phantom{$2^{2^{2^2}}$}$\card{V(H)}\backslash \cc(\overline{H})$& 1 & 2 & 3 & 4\\
    \hline
    1 & $K_1$ & \cellcolor[gray]{0.75}&\cellcolor[gray]{0.75} &\cellcolor[gray]{0.75}\\
    \hline
    2 & $2 \cdot K_1$ & $K_2$& \cellcolor[gray]{0.75}&\cellcolor[gray]{0.75}\\
    \hline
    3 & $K_2 + K_1$ & $P_3$& $K_3$& \cellcolor[gray]{0.75}\\
    \hline
    4 & $P_4$ & $K_3^+$& $K_4^-$& $K_4$\\
    \hline
    5 & \cellcolor[gray]{0.75} \remref{e:r3}& $\mathrm{Gem}$& $\htg$& \cellcolor[gray]{0.75} \remref{e:r6}\\
    \hline
    6 & \cellcolor[gray]{0.75} \remref{e:r3}& \cellcolor[gray]{0.75}
    \remref{e:r3}& \cellcolor[gray]{0.75}
    \remref{e:r5}&\cellcolor[gray]{0.75} \remref{e:r6}\\
    \hline
    7 & \cellcolor[gray]{0.75} \remref{e:r3}& \cellcolor[gray]{0.75}
    \remref{e:r3}& \cellcolor[gray]{0.75}
    \remref{e:r5}&\cellcolor[gray]{0.75} \remref{e:r6}\\
    \hline
  \end{tabular}
  \caption{If $(\excl(H), \linm)$ is a wqo, then $H$
    belongs to this table.}
  \label{t:h}
\end{table}

Observe that we have
\begin{itemize}
\item $K_1 \linm 2 \cdot K_1 \linm K_2 + K_1 \linm P_4 \linm \gem$;
\item $K_2 \linm P_3 \linm K_3^+ \linm \gem$;
\item $K_3 \linm K_4^- \linm \htg$; and
\item $K_4 \linm \htg$.
\end{itemize}
This concludes the proof.

\end{proof}

We are now ready to give the proof of \autoref{t:dich}.

\begin{proof}[Proof of \autoref{t:dich}]
  If $H \not \linm \gem$ and $H \not \linm \htg$, then by
  \autoref{l:prethe} $(\excl(H), \linm)$ is not a wqo.
  On the other hand, by \autoref{t:exclh3-wqo} and
  \autoref{t:exclgem-wqo} we know that both 
  $\excl(\htg)$ and $\excl(\gem)$ are wqo by~$\linm$.
  Consequently, by \autoref{r:exclincl}, $(\excl(H), \linm)$ is wqo as soon as
  $H\linm \gem$ or $H \linm \htg$.
\end{proof}

\section{Graphs not
  containing~\texorpdfstring{$\htg$}{K4\^{}}}
\label{sec:wqohtg}

The main goal of this section is to provide a proof to
\autoref{t:exclh3-wqo}. To this purpose, we first prove in
\autoref{sec:dh3} that graphs of $\excl(\htg)$ admit a simple
structural decomposition (\autoref{t:dec-h3}). This structure is then used in
\autoref{sec:wqoh3} to show that graphs of $\excl(\htg)$ are
well-quasi-ordered by the relation~$\linm$.

\subsection{A decomposition theorem for \texorpdfstring{$\excl(\htg)$}{Excl(K4')}}
\label{sec:dh3}

This section is devoted to the proof of \autoref{t:dec-h3}. This
theorem states that every graph in the class $\excl(\htg)$, either
does not contain $K_4$ as induced minor, or is a subdivision of
some small graph, or can be partitioned into ``simple'' parts.

The proof is split into several parts. Recall that a \emph{proper $K_4$-subdivision} in a graph $G$ is a $K_4$-subdivision that does not uses all vertices.
First we show (\autoref{s:find}) that if $G$ is neither a subdivision of a small graph, nor a wheel, nor a $K_4$-subdivision free graph, then $G$ has a proper $K_4$-subdivision~$S$. We then deduce properties about vertices of $G$ that do not belong to $S$ (\autoref{sec:neighseub}). In \autoref{sec:smallk4sub}, we handle the case where $S$ is in fact a $K_4$-subgraph and show that in the other cases, we can focus on the case where $S$ is a 3-wheel. The last part of the proof is addressed in \autoref{sec:largesub}.

\subsubsection{Finding a proper \texorpdfstring{$K_4$}{K4}-subdivision}
\label{s:find}
  
In this section we show that, unless $G$ has a simple structure, it contains a proper $K_4$-subdivision.
The proof of the following easy lemma is left to the reader.
\begin{lemma}\label{l:k33prism}
  If $G$ can be obtained by adding an edge between two vertices of a
  $K_{3,3}$-subdivision (resp.\ a prism-subdivision), then $G$ has a proper $K_4$-subdivision. 
\end{lemma}

\begin{lemma} \label{l:notproperk4}
  If graph $G$ contains a $K_4$-subdivision, then one of the following holds:
  \begin{enumerate}[(i)]
  \item $G$ has a proper $K_4$-subdivision, or
  \item \label{npk:w} $G$ is a wheel, or
  \item \label{npk:s} $G$ is a subdivision of $K_4$, $K_{3,3}$, or the prism.
  \end{enumerate}

\end{lemma}
\begin{proof}
	Looking for a contradiction, let $G$ be a counterexample with
        the minimum number of vertices and, subject to that, the
        minimum number of edges. Let $S$ be a $K_4$-subdivision in $G$.
        As $G$ has no proper $K_4$-subdivision, $S$ is a spanning subgraph. Besides, $G$ is not a $K_4$-subdivision so there is an edge $e \in E(G) \setminus E(S)$.
	Notice that since the minimum degree of $K_4$ is 3,
        contracting an edge incident with a vertex of degree 2 in $G$
        would yield a smaller counterexample. Therefore $G$ has minimum
        degree at least~3.
        Let $G' = G \setminus \{e\}$. This graph clearly contains~$S$. By minimality of $G$, the graph $G'$ is either a wheel, or a
        subdivision of a graph among $K_4$, $K_{3,3}$, and the prism. Observe that $G'$ cannot have a proper $K_4$-subdivision because it would also be a proper $K_4$-subdivision in~$G$.

        \smallskip
        \noindent\textit{First case:} $G'$ is a wheel.
        Let $C$ be the cycle of the wheel and let $r$ be its
        center. In $G$ the edge $e$ does not have $r$ as an endpoint,
	because
        otherwise $G$ would also be a wheel. Therefore $e$ is
        incident with two vertices of $C$. Let $P$ and $P'$ be the two
        subpaths of $C$ whose endpoints are the endpoints of~$e$.
        Observe that none of $P$ and $P'$ contains more than two
        neighbors of $r$.  
        Indeed, if, say, $P$ contained at least
        three neighbors of $r$, then the subgraph of $G$ induced by
        the vertices of $P$, $e$, an $r$ would contain a $K_4$-subdivision,
        hence contradicting the fact that $G$ has no proper
        $K_4$-subdivision.

        Therefore $G$ is the cycle $C$ with exactly one chord, $e$,
        and the vertex $r$ which has at most 4 neighbors on
        $C$. Because $G$ has minimum degree at least 3, it has at most
        7 vertices.
        We can easily check that if $r$ has three neighbors on $C$, then either one of $P$ and $P'$ contain exactly one of them, in which case $G$ is a
        subdivision of the prism, or both contain two of them (one being contained in both $P$ and $P'$) and $G$ is a wheel (with a center which is the neighbor of $C$ lying on both $P$ and $P'$).
        If $r$ has four neighbors on $C$, the interior of $P$ and $P'$
        must each contain two of them according to the above
        remarks. The deletion of any neighbor of $r$ in this graph
        yields a $K_4$-subdivision of non-subdivision vertices $r$
        and the remaining neighbors.
        Observe that both cases contradict the assumptions made on~$G$.

        \smallskip
        \noindent\textit{Second case:} $G'$ is a subdivision of $K_4$,
          or $K_{3,3}$, or the prism. In the two latter cases the result follows by \autoref{l:k33prism}.
        We therefore assume that $G'$ is a subdivision of $K_4$.
        A \emph{branch} of $S$ is a maximal path, the internal vertices of
        which have degree two (in the subgraph $S$). Notice that every branch of $S$ is
        chordless in, otherwise one could shortcut it and thus find a
        proper $K_4$-subdivision in $G$.
        In the case where the endpoints of $e$ belong to the interior of two
        different branches, then it is easy to see that $G$ is a
        prism-subdivision if these branches share a vertex and a
        subdivision of $K_{3,3}$ otherwise. Let $\{x,y,z,t\}$ be the
        non-subdivision vertices of the $K_4$-subdivision. We denote
        by $B_{s,t}$ (for $s,t \in \{x,y,z,t\}$) the branch ending at
        vertices $s$ and $t$. Finally, let us assume that
        the one endpoint of $e$ is a non-subdivision vertex, say $x$, and the other one, that we
        call $u$, is a subdivision vertex of a branch, say $B_{y,z}$.
        If $X$ is the set of interior vertices of one of $B_{x,y}$,
        $B_{x,z}$, or $B_{x,t}$, then the
        graph $G\setminus X$ has a $K_4$ subdivision of non-subdivision
        vertices~$x,u,z,t$, $x,y,y,t$ or $x,y,u,z$ respectively. In
        this case $G$ has a proper $K_4$-subdivision. If none of $B_{x,y}$,
        $B_{x,z}$, and $B_{x,t}$ has internal vertices, then $G$ is a
        wheel of center~$x$.

In all the possible cases we reached the contradiction we were looking for. This concludes the proof.
\end{proof}

\subsubsection{On the neighbors of proper \texorpdfstring{$K_4$}{K4}-subdivisions}
\label{sec:neighseub}

The outcomes \itemref{npk:w} and \itemref{npk:s} of \autoref{l:notproperk4} match possible outcomes of \autoref{t:dec-h3}, as noted in its proof in the end of \autoref{sec:largesub}, therefore we now focus on the case where $G$ has a proper $K_4$-subdivision. The following lemma describes a structure that forces $\htg$-induced minors and will be used to deduce properties of $\htg$-induced minor-free graphs.

\begin{lemma}\label{l:h3struct}
  If $G$ contains as induced minor any graph $H$ consisting of:
  \begin{itemize}
  \item a $K_4$-subdivision $S$;
  \item an extra vertex $x$ linked by exactly two paths $L_1$ and
    $L_2$ to two distinct vertices $s_1,s_2 \in V(S)$, where
    the only common vertex of $L_1$ and $L_2$ is~$x$;
	\label{e:twop}
\item and possibly extra edges between the vertices of $S$, or
  between $L_1$ and $L_2$, or between the interior of the paths and~$S$,
  \end{itemize}
  then $\htg \linm G.$
\end{lemma}

\begin{proof}
  Let us call $V=\{v_1, v_2, v_3, v_4\}$ the non-subdivision
  vertices of $S$, i.e.~vertices corresponding to vertices of~$K_4.$ We
  present here a sequence of edge contractions
  in $H$ leading to $\htg$.
  As long as there is a path between two elements of $V \cup \{s_1, s_2, x\}$,
  internally disjoint with this set, we contract the whole path to a single
  edge.   
  
  Once we cannot apply this contraction any more, we end up with a graph 
  that has two parts: the $K_4$-subdivision with at most 
  2 subdivisions (with vertex set $V \cup \{s_1, s_2\}$) and the vertex $x$,
  which is now only adjacent to $s_1$ and $s_2.$

  \smallskip
  \noindent \textit{First case:} $s_1, s_2 \in V.$ The graph $H$ is
  isomorphic to $\htg$: it is $K_4$ plus a vertex of degree~two.

  \smallskip
  \noindent \textit{Second case:} $s_1 \in V$ and $s_2 \not
  \in V$ (and the symmetric case). As vertices of $V$ are the only
  vertices of $H$ that have degree 3 in $S,$ $s_2$ is of degree
  2 in $S$ (it is introduced by subdivision). The contraction of the edge
  between $s_2$ and one of its neighbors in $S$ that is different from
  $s_1$ leads to the first case.

  \smallskip
  \noindent \textit{Third case:} $s_1,s_2 \not \in V.$ As in second
  case, these two vertices have degree two in $S.$ Since no two
  different edges of $K_4$ can have the same endpoints, the
  neighborhoods of $s_1$ and $s_2$ have at most one common vertex.
  Then for every $i\in \{1,2\}$ there is a neighbor $t_i$ of $s_i$ that
  is not adjacent to $s_{3-i}.$ Contracting the edges $\{s_1,t_1\}$
  and $\{s_2,t_2\}$ leads to the first case.
\end{proof}

\begin{corollary}[from \autoref{l:h3struct}]\label{c:deg3} 
  Let $G \in \excl(\htg)$ be a 2-connected graph containing a proper
  $K_4$-subdivision $S$. For every vertex $x \in V(G) \setminus V(S),$
  $N_S(x) \geq 3.$
\end{corollary}

\begin{proof}
  As $G$ is 2-connected, we can find when $N_S(x) \leq 2$ two paths from $x$ to $S$ that satisfy the conditions of \autoref{l:h3struct}.
  Therefore, $N_S(x) \geq 3.$
\end{proof}

\subsubsection{Small \texorpdfstring{$K_4$}{K4}-subdivisions}
\label{sec:smallk4sub}

We handle separately the case where $G$ contains a subgraph isomorphic to~$K_4$.

\begin{lemma}\label{l:smallk4}
  Let $G \in \excl(\htg)$ be a 2-connected graph. If $G$ has a subgraph $S$ that is isomorphic to $K_4$, then $G\setminus V(S)$ is
    complete multipartite.
  \end{lemma}

  \begin{proof}
    Let us first show that for every non-adjacent $u,v \in V(G)\setminus V(S)$, $N_S(u) =
    N_S(u)$.
  According to \autoref{c:deg3}, each
  of $u$ and $v$ has at least 3 neighbors in $S$, hence $3 \leq |N_S(u)|,
  |N_S(v)| \leq 4$. If $\left |N_S(u) \right | = \left |N_S(v)
  \right| = 4$, then $N_S(u) = N_S(v)$ and we are done. Towards a
  contradiction, we assume $N_S(u) \neq N_S(v)$. In the case where
  $|N_S(u)| = |N_S(v)| = 3$ (resp.\ $|N_S(u)| = 3$ and $|N_S(u)| = 4$,
  or the other way around), we have $\htg \linm G$ as
  depicted on \autoref{fig:smallk4}.(a), a contradiction.
  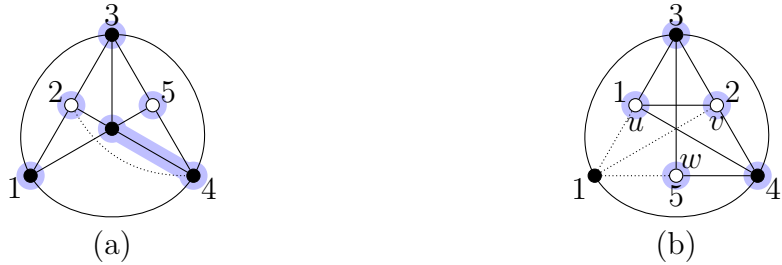
\begin{figure}[h]
  \centering
  \begin{tikzpicture}[every node/.style = black node, scale = 1.25]
    \begin{scope}[xshift = 0cm]
      \draw (0, -1.25) node[normal] {(a)};
      \fill[color =blue!25] (90:1) circle (0.15cm) (330:1) circle (0.15cm) (210:1) circle (0.15cm) (0, 0) circle (0.15cm)
      (30:0.5) circle (0.15cm) (150:0.5) circle (0.15cm);
      \draw[color = blue!25, line width =0.25cm] (0, 0) -- (330:1);
      \draw (90:1) node (a) {} to[bend right=60, label=90:3] (210:1) node[label=201:1] (b) {} to[bend right=60] (330:1) node[label=330:4] (c) {} to[bend right=60] (a);
      \draw (0,0) node (d) {} (a) -- (d) (b) -- (d) (c) -- (d);
      \draw (150:0.5) node[white node, label = 150:2] (v) {};
      \draw (30:0.5) node[white node, label = 30:5] (u) {};
      \draw (u) -- (a) (u) -- (c) (u) -- (d);
      \draw (v) -- (a) (v) -- (b) (v) -- (d);
      \draw[densely dotted] (v) to[bend right] (c);
    \end{scope}
    \begin{scope}[xshift = 6cm]
      \draw (0, -1.25) node[normal] {(b)};
      \fill[color =blue!25]
      (90:1) circle (0.15cm)
      (330:1) circle (0.15cm)
      (-90:0.5) circle (0.15cm) node[normal, color = black, label=-90:\textcolor{black}{5}] {}
      (30:0.5) circle (0.15cm) node[normal, color = black, label=30:\textcolor{black}{2}] {}
      (150:0.5) circle (0.15cm) node[normal, color = black, label=150:\textcolor{black}{1}] {};
      \draw (90:1) node[label=90:3] (a) {} to[bend right=60] (210:1) node[label=210:1] (b) {} to[bend right=60] (330:1) node[label=330:4] (c) {} to[bend right=60] (a);
      \draw (30:0.5) node[white node, label=-90:$v$] (u) {} (150:0.5) node[white node, label =-90:$u$] (v) {} (-90:0.5) node[white node, label = 45:$w$] (w) {};
      \draw (u) -- (a) (u) -- (c);
      \draw (v) -- (a) (v) -- (c);
      \draw (w) -- (a) (w) -- (c);
      \draw (u) -- (v);
      \draw[densely dotted] (u) -- (b) (v) -- (b) (w) -- (b);
    \end{scope}
  \end{tikzpicture}
  \caption[Finding $\htg$ as induced minor in the proof of
  \autoref{l:smallk4}]{Finding $\htg$ as induced minor in the proof of
    \autoref{l:smallk4}. Vertices of $S$ are black and $u,v,w$ are white. Dotted lines represent edges that may be
  present or not. The numbers indicate which vertices of $\htg$ (following the convention of \autoref{fig:h123}) correspond to the subsets of vertices depicted in blue.}
  \label{fig:smallk4}
\end{figure}
  Let us now show that $G \setminus V(S)$ is complete multipartite.
  A graph is complete multipartite iff it does not contain $K_1 + K_2$
  as induced subgraph. Towards a contradiction, we therefore assume
  that there are three vertices $u,v,w \in V(G) \setminus V(S)$ such
  that $\{u,v\}$ is the only edge in $G[\{u,v,w\}]$. 
  According to the paragraph above applied to
  $u$ and $w$ and then to $w$ and $v$, we have $N_S(u) = N_S(v) =
  N_S(w)$. As noted above, each of $u$, $v$ and $w$ have at least
  three neighbors on $S$. In this case again we are able to find $\htg$ as
  an induced minor (in fact, as an induced subgraph), as depicted in
  \autoref{fig:smallk4}.(b). This is a contradiction, hence $G
  \setminus (V(S) \setminus s)$ is complete multipartite.
\end{proof}

Notice that the conclusion of \autoref{l:smallk4} is an outcomes of \autoref{t:dec-h3}.
We now show that some minimum $K_4$-subdivision is a 3-wheel. This will allow us to focus on the case where $S$ is a 3-wheel.
\begin{lemma}\label{l:wheel} 
  If $G \in \excl(\htg)$ is a 2-connected graph with a proper
  $K_4$-subdivision, then some $K_4$-subdivision of $G$ with the minimum number of vertices is a 3-wheel.
\end{lemma}

\begin{proof}
  We show that for every proper $K_4$-subdivision $S$ of $G$, we can
  find a 3-wheel with at most the same number of vertices. Let $x
  \in V(G) \setminus V(S)$ and let $V$ be as in the
  proof of \autoref{l:h3struct}. We call two neighbors of $x$ in
  $S$ \emph{equivalent} if they lie on the same path between two
  elements of $V$. Intuitively, equivalent vertices correspond to the
  same edge of~$K_4$. By \autoref{c:deg3}, we can assume~$\card{N_S(x)} \geq 3$.

  We start with the case where no cycle of $S$ contains three neighbors of $x.$ 
  Notice that for every pair of edges of $K_4$, there is a
  cycle using these edges. This implies that $x$ does neither have two
  equivalent neighbors, nor a neighbor in~$V$.
  Besides, for every choice of four edges of $K_4$, there is a
  Hamiltonian cycle containing
  three of them. We deduce $\card{N_S(x)} = 3$. Let us
  consider the induced minor $H$ of $G[V(S) \cup \{x\}]$ obtained by
  iteratively contracting
  all edges that are incident with at most one vertex of~$V \cup
  N_S(x) \cup \{x\}.$
  By the above remark and as $x$ does not have three neighbors on a
  cycle, there is a vertex of $V(H) \setminus \{x\}$ adjacent to the
  three neighbors of~$x$. Contracting two of the edges incident with this vertex merges
  two neighbors of $x$ and the graph we obtain is a $K_4$ subdivision
  (corresponding to~$S$) together with a vertex of degree 2
  (corresponding to~$x$). By \autoref{l:h3struct}, we have $\htg \linm G$, a contradiction.

  In the remaining case, $S$ has a cycle containing three neighbors of
  $x$. Let $C$ be such a cycle with the minimum number of
  vertices. Notice that $V(S) \setminus V(C)$ is then not empty, hence
  $C \cup \{x\}$ is the desired wheel and does not have more vertices
  than~$S$.
\end{proof}

\subsubsection{Dealing with proper \texorpdfstring{$K_4$}{K4}-subdivisions}
\label{sec:largesub}

In the sequel, we deal with a graph $G$ that is 2-connected and has a proper $K_4$-subdivision, but contains neither
$\htg$ as induced minor, nor $K_4$ as subgraph. Let us consider a subgraph $S$ of $G$ that is a 3-wheel with the minimum number of vertices and, subject to this requirement, has a minimum number of chords (i.e. edges of $E(G) \setminus E(S)$ with both endpoints in $V(S)$).
We denote by $C$ the cycle of this 3-wheel, by $r$ its center, and set $R = G\setminus V(C)$.
These assumptions are implicit in the following lemmas.

As $S$ is not necessarily an induced subgraph of $G$, the cycle $C$ may have chords. We consider this case hereafter.


\begin{lemma}\label{l:onechordnew}
  If $C$ is not an induced cycle in $G$, then $G[V(S)]$ is the prism.
\end{lemma}

\begin{proof}
  Let $u, v, w \in N_C(r)$ be three distinct neighbors of $r$ in
  $C$ and let $C_u$ be the path of $C$ between $v$ and $w$ that
  does not contain the vertex $u$, and similarly for $C_v$
  and~$C_w$.
  Let us assume that $C$ has a chord $\{x,y\}$. Observe that $x$
  and $y$ cannot both belong to $C_l$ for some $l \in \{u,v,w\}$, as
  the deletion of any interior vertex of $\patht{x}{C_l}{y}$ (which
  exist as $\{x,y\}$ is a chord of $C$) would
  leave a $K_4$-subdivision in $S$, contradicting its minimality (see
  \autoref{fig:forbconf}.(a)). Therefore, $x$ and $y$ belong to different
  $C_l$'s, say without loss of generality that $x \in V(C_w)$ and~$y
  \in V(C_v)$.
\begin{figure}[ht]
  \centering
  \begin{tikzpicture}[every node/.style = black node]

    \newcommand{\cycleC}[1]{
      \draw node[label = -90:$r$] (r) at (0,0) {};
      \draw[decorate,decoration={zigzag, amplitude = 0.03cm}] (r) circle (1cm);
      \node[normal, draw = none, label = -60:$C$] at (-60:1) {};
      \foreach \x / \l in {90 / $u$, 210 / $v$, -30 / $w$} {
        \draw (r) -- (\x:1) node[label = \x:\l] (\l) {};
      }
    }
    \begin{scope}[xshift = -5cm]
      \draw (0, -1.5) node[normal] {(a)};
      \cycleC{}
      \draw (120:1) node[label = 120:$x$] {} to (180:1) node[label =
      180:$y$] {};
      \draw (150:1) node[fill = white] {};
    \end{scope}
    \begin{scope}
      \draw (0, -1.5) node[normal] {(b)};
      \cycleC{}
      \draw (150:1) node[label = 120:$x$] {} to[bend left] (-30:1);
      \draw (30:1) node[fill = white] {};
    \end{scope}
  \end{tikzpicture}
  \caption[When $C$ has a chord.]{When $C$ has a chord. In these examples, there is a $K_4$-subdivision that does not use the white vertex, which contradicts the minimality of~$S$. Zigzag lines depict paths with at least one edge.}
  \label{fig:forbconf}
\end{figure}
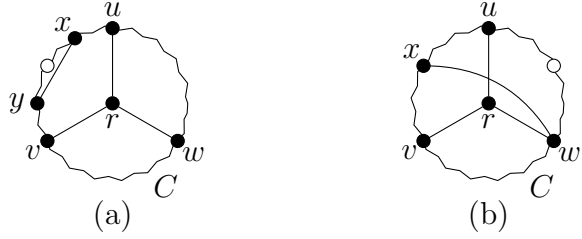

  \smallskip
  \noindent \textit{First case:} $y  = w$ and $x$ belongs to the interior of~$C_w$.
  Observe that if one of $C_u$ or $C_v$ has an internal vertex
  $z$, then $S \setminus z$ still has a $K_4$-subdivision (see
  \autoref{fig:forbconf}.(b)), which would contradict the minimality
  of~$S$. Hence each of $C_u$ and $C_v$ is reduced to an edge. Now,
  notice that if $r$ is adjacent to an internal vertex of $C_w$, then again one can find a
  smaller $K_4$-subdivision, for instance by deleting~$u$.
  The path $C_w$ together with $r$ forms the cycle $C'$ of a
  4-wheel of center~$w$. As noticed above, $C_w$ is chordless and $r$
  has no neighbors on this path. Therefore $C'$ is chordless: it has
  fewer chords than~$C$. This contradicts the definition of~$S$, hence
  this case is not possible.

  \smallskip
  \noindent \textit{Second case:} $x$ and $y$ are interior
  vertices of $C_w$ and $C_v$, respectively. We first show that $|C| =
  5$. As in the previous case, it is easy to see that if one of
  $\patht{u}{C_w}{x}$, $\patht{x}{C_w}{v}$, $\patht{u}{C_v}{y}$,
  $\patht{y}{C_v}{w}$ or $C_u$ has an internal vertex $z$, then
  $S\setminus z$ is not $K_4$-subdivision free. As this contradicts
  the definition of $S$, we deduce that each of them is reduced to an
  edge, proving that~$|C|=5$. In the light of the previous case, no
  endpoints of a chord of $C$ can belong to $\{u,v,w\}$. As $C$ has 5
  vertices, it has only one chord. This concludes the proof.
\end{proof}

\begin{corollary}\label{c:prism}
  If $C$ has a chord, then $G$ is the prism.
\end{corollary}

\begin{proof}
  From \autoref{l:onechordnew}, we get that $S$ is the prism. Observe
  that for every choice of three vertices of the prism, there is a
  cycle of length at most 4 containing them. Let $v \in V(G)
  \setminus V(S)$. The vertex $v$ has at least 3 neighbors in $S$ (\autoref{c:deg3}),
  thus it is the center of a 3-wheel on size at most 5, using a cycle as
  mentioned above. This contradicts the minimality of $S$, the prism, which has six vertices. Hence
  we deduce that $V(G) \setminus V(S)$ is empty: $G$ is the prism.
\end{proof}

The case where $C$ has chords being fixed, we assume in the remaining
of this section that $C$ is a chordless cycle.

\begin{lemma}\label{l:neigh2}
  If for some $t \in V(G)\setminus V(S)$, $N_C(t) \leq 2$, then $|C| = 4$
  and every $t'\in V(G)\setminus V(S)$ such that $N_C(t') \leq 2$ has the same neighborhood on
  $S$ as $t$.
\end{lemma}

\begin{proof}
Let $t \in V(G)
  \setminus V(S)$. As $N_S(t) \geq 3$ (\autoref{c:deg3}) and $N_C(v)
  \leq 2$, we deduce that $r \in N_S(t)$ and
  $|N_S(t)|=3$.
  Let $x,y$ be the two neighbors of $t$ on $C$.
  We define $u$, $v$, $w$, $C_u$, $C_v$, and $C_w$ as in the proof of \autoref{l:onechordnew}.
  We consider different
cases according to the positions of $x$ and $y$.

\smallskip
\noindent \textit{First case:} both $x$ and $y$ belong to one of $C_u,
C_v, C_w$, say $C_u$, without loss of generality, at least one of them being in the
interior of the path. Then contracting the subpath of $C_u$ that links
$x$ to $y$ yield a graph that has a vertex of degree 2, $t$, with
exactly two neighbors on a $K_4$-subdivision, contradicting
\autoref{c:deg3}. Hence this case is not possible.

\smallskip
\noindent \textit{Second case:} $x,y \in \{u,v,w\}$, say $x = u$ and
$y = v$, without loss of generality. Observe that if one of $C_v$ or
$C_u$ has an interior vertex, then the graph induced by
$r, t$ and $V(C_w)$ has a $K_4$-subdivision that is smaller than $S$,
a contradiction. We deduce that each of these paths is reduced to an
edge. In order to show the same thing for~$C_w$, we assume that
there is an interior vertex $z$ to~$C_w$. Note that
$G[u,v,w,r,t]$ has a $K_4$-subdivision. As $C$ is induced, $z$ has degree 2 in $S$.
Then it has at most two neighbors in the
aforementioned $K_4$-subdivision, a contradiction to
\autoref{c:deg3}. Therefore, $C_w$ is an edge. We deduce that
$|S|=4$. That is, $S$ is a $K_4$-subgraph, where we assumed the
opposite. Consequently, this case is not possible either.

\smallskip
\noindent \textit{Third case:} $x$ and $y$ belong to the interior of
two different paths among $C_u$, $C_v$, and $C_w$, say without loss of
generality $C_u$ and $C_v$, respectively. Then by contracting the subpath
of $C_u$ linking $x$ to $w$, we reach the first case.

\smallskip
\noindent \textit{Fourth case:} for some $z \in \{u,v,w\}$, $x$
belongs to the interior of $C_{z}$ and $y =
z$. Without loss of generality we assume that $z=w$.
As in the previous lemmas, it is easy to check that if one of $C_u$,
$C_v$, $\patht{u}{C_w}{x}$, and $\patht{x}{C_w}{v}$ has an interior
vertex, then deleting it and deleting $v$, $u$, $u$, or $v$,
respectively, does not make $G[V(S) \cup \{t\}]$ $K_4$-subdivision
free. This contradicts the minimality of $S$, hence each of these
paths is reduced to an edge. We deduce $|C|=4$. In the light of the
previous remarks and as $|C|=4$, the only possible neighbors for a
vertex $t'$ as in the statement of the lemma are $x$ and $y$, which concludes
the proof.
\end{proof}

We can now focus on vertices of $V(G) \setminus V(S)$ that are
adjacent to at least three vertices of~$C$.
\begin{lemma}\label{l:p4} 
  If some $s \in V(G) \setminus V(S)$ satisfies $|N_C(s)| \geq 3$, then $N_C(s) = N_C(r)$.
\end{lemma}

\begin{proof}
  Towards a contradiction, we assume that some $u \in V(C)$ is adjacent to $r$ but not to~$s$. As $s$ and $r$ play a symmetric role, this is the only case to consider.

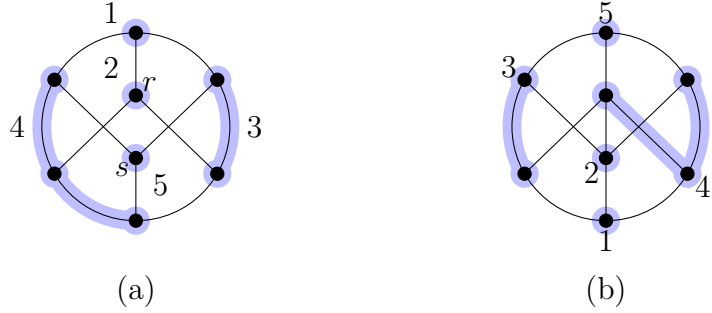
\begin{figure}[h]
  \centering
  \begin{tikzpicture}[every node/.style = black node, scale = 1.25]
    \def \radius {1}
    \begin{scope}[xshift = 0cm]
      \draw (0, -1.75) node[normal] {(a)};
      \begin{scope}[color = blue!25, every node/.style = {normal,
          minimum size = 0, inner sep=3pt}]
        \fill (90:1) circle (0.15cm)
        node[label=170:\textcolor{black}{1}] {};
        \fill (150:1) circle (0.15cm);
        \fill (210:1) circle (0.15cm);
        \fill (270:1) circle (0.15cm);
        \fill (330:1) circle (0.15cm);
        \fill (30:1) circle (0.15cm);
        \fill (90:1/3) circle (0.15cm)
        node[label=135:\textcolor{black}{2}] {};
        \fill (90:-1/3) circle (0.15cm)
        node[label=-45:\textcolor{black}{5}] {};
    \draw[line width =0.25cm] (150:1) arc (150:270:1) (330:1) arc
    (330:390:1);
    \draw (180:1) node[label=180:\textcolor{black}{4}] {};
    \draw (0:1) node[label=0:\textcolor{black}{3}] {};
  \end{scope}
      \draw (0,0) circle (1cm);
    \foreach \i / \l in {0 / u_1, 1 / v_3, 2 / u_3,
      3 / v_2, 4 / u_2, 5 / v_1}{
      \draw (90+\i*60:\radius) node {};
    }
    \node[label=30:$r$] (s) at (90:{\radius / 3}) {};
    \node[label=210:$s$] (t) at (-90:{\radius / 3}) {};
    \foreach \i in {0, 2, 4}{
      \draw (s) -- (90 + \i * 60:\radius);
      \draw (t) -- (150 + \i*60:\radius);
    }
    \end{scope}
    \begin{scope}[xshift = 5cm]
      \draw (0, -1.75) node[normal] {(b)};
      \fill[color = blue!25] (90:\radius/3) circle (0.15cm);
      \fill[color = blue!25] (90:-\radius/3) circle (0.15cm);
      \draw[color = blue!25, line width =0.25cm] (150:1) arc (150:210:1)
      (90:\radius/3) -- (330:1) arc (330:390:1);
      \foreach \i / \l in {0/$5$,1/$3$,2/{},3/$1$,4/$4$,5/{}}{
        \fill[color = blue!25] (90+\i*60:\radius) circle (0.15cm);
        \draw (90+\i*60:\radius) node[label = 90+\i*60:\l] (M\i) {};
      }
    \node (s) at (90:{\radius / 3}) {};
    \node[label=-135:$2$] (t) at (-90:{\radius / 3}) {};
    \draw (0,0) circle (1cm);
    \foreach \i in {0, 2, 4}{
      \draw (s) -- (90 + \i * 60:\radius);
      \draw (t) -- (150 + \i*60:\radius);
    }
    \draw (s) -- (t);
    \end{scope}
  \end{tikzpicture}
  \caption[Models of $\htg$ in the proof of \autoref{l:p4}.]{Models of $\htg$ in the proof of \autoref{l:p4}. The numbers indicate which vertices of $\htg$ (following the convention of \autoref{fig:h123}) correspond to the subsets of vertices depicted in blue.\label{f:graphh}}
\end{figure}

  Let $v,w$ (both distinct from $u$) and $u', v', w'$ be neighbors of $r$ and
  $s$ on $C$, respectively. We consider
  the graph $H$ obtained from $G[V(C) \cup \{r,s\}]$ by iteratively
  contracting every edge of $C$ that is not incident with two vertices of $\{u,
  v,w,u',v', w'\}.$ This graph is an induced
  cycle (as $C$ is induced) on at most 6 vertices, that we call $C'$,
  plus the two vertices $r$ and
  $s$ that have at least three neighbors each on $C$.
  Observe that while two neighbors of $s$ are adjacent and are not both
  neighbors of $r$, we can contract the edge between them and decrease
  by one the degree of $s$, without changing degree of~$r$. 
  If the degree of $s$ reaches two by such means, then by
  \autoref{l:h3struct}, $\htg \linm H$, a contradiction.
  We can thus assume that every vertex of $C'$ adjacent to a neighbor
  of $s$ is a neighbor of~$r$. This is also true when $r$ and $s$ are
  swapped since this argument can be applied to $r$ too. This
  observation implies that $N_S(r) \cap N_S(s) = \emptyset$ (as $u$
  is adjacent to $r$ but not to~$s$, none of its neighbors on $C$ can
  be adjacent to $r,$ and so on along the cycle)
  and that the neighbors of $r$ and $s$ are alternating on $C'$.
  Without loss of generality, we suppose that $C' = uu'vv'ww'$.
  \autoref{f:graphh}.(a) and \autoref{f:graphh}.(b) shows how a model of $\htg$ can then be found
  in this case, depending whether $\{r,s\} \in E(G)$, respectively, a contradiction.
\end{proof}

\begin{lemma}\label{l:cm} 
  $G \setminus V(S)$ is complete multipartite.
\end{lemma}

\begin{proof}
  Let us consider the graph obtained from $G$ by contracting $S$ to
  $K_4$. Observe that this does not impact the adjacencies in
  $G\setminus V(S)$. The result then follows from \autoref{l:smallk4}.
\end{proof}

\begin{lemma}\label{l:stable}
Either $|C| \leq 4$ or $V(G)\setminus V(C)$ is an independent set.
\end{lemma}

\begin{proof}
Assuming that $|C| \geq 5$, let us show that $V(G) \setminus V(C)$ is an independent set.
By \autoref{l:neigh2} and \autoref{l:p4}, the vertices of $V(G) \setminus V(C)$ all have the same neighborhood on $C$, which has size at least~3.
Towards a contradiction, let us assume that there exist two adjacent vertices $x,y \in V(G) \setminus V(C)$.
We define $u$, $v$, $w$, $C_u$, $C_v$, and $C_w$ as in the proof of \autoref{l:onechordnew}.
Observe that the graph induced by $x$, $y$, and $C_w$ contains a $K_4$-subdivision.
We deduce that none of $C_u$ and $C_v$ contains an internal vertex, otherwise the deletion of this vertex and $w$ would produce a graph violating the minimality of~$S$. Symmetrically, $C_w$ has no internal vertex. Hence $|C| = 3$, a contradiction. This proves that  $V(G)\setminus V(C)$ is an independent set.
\end{proof}

We are now ready to prove \autoref{t:dec-h3}.

\begin{proof}[Proof of \autoref{t:dec-h3}]
  Let $G \in \excl(\htg)$ be a 2-connected graph.
  If $G$ does not contain a $K_4$-subdivision, or if $G$ is a subdivision of
  $K_4$, $K_{3,3}$, then we are done (outcomes \itemref{h3-1} or \itemref{h3-2} of the theorem). If
  $G$ contains a $K_4$-subdivision but not a proper one, from
  \autoref{l:notproperk4} we get that $G$ is a subdivision of one of
  $K_4$, $K_{3,3}$, or the prism, in which case the theorem holds
  (outcome \itemref{h3-2}), or that $G$ is a wheel, which has a trivial
  partition $(V(G)\setminus \{r\} ,\{r\})$ (where $r$ is the center of the wheel) satisfying item \itemref{h3-4} of the statement of
  the theorem.

  Therefore we assume that $G$ does not fall in one of the
  aforementioned cases. By \autoref{l:notproperk4}, $G$ then has a 
  proper $K_4$-subdivision. If $G$ has a $K_4$-subgraph $S$, then
  $(V(S), V(G) \setminus V(S)$ is a partition satisfying item
  \itemref{h3-3} of the desired statement, according to \autoref{l:smallk4}.

  We now focus on the case where $G$ has no $K_4$-subgraph and we
  consider a minimal 3-wheel as defined at the beginning of
  \autoref{sec:largesub}, using the same notation.
  The case where $C$ is not induced is not possible as we assume that
  $G$ is not a prism (cf.\ \autoref{c:prism}).
  If $V(G) \setminus  V(S)$ has a vertex $t$ such that $|N_C(t) | =2$,
  then, by the virtue of \autoref{l:neigh2} and \autoref{l:cm}, $|S|  = 5$ and the
  partition $(V(S), V(G)\setminus V(S))$ suits the requirements of
  \itemref{h3-3}. 
  Otherwise, every vertex $t \in V(G)\setminus V(S)$ satisfies $|N_C(t)| \geq 3$. Then, by \autoref{l:p4} and \autoref{l:stable}, $G \setminus V(S)$ is an independent set and its vertices have the same neighborhood on $C$. Therefore, the partition $(V(C),V(G)\setminus V(C))$ satisfies the conditions of outcome \itemref{h3-4}.
\end{proof}

\subsection{From a decomposition theorem to well-quasi-ordering}
\label{sec:wqoh3}

This section is devoted to the proof of \autoref{t:exclh3-wqo}.
We define the two following classes of graphs:
\begin{itemize}
\item $\wm$ is the class of the graphs that admit a partition $(W,M)$ of their vertex set such that $W$ induces a wheel on at most 5 vertices and $M$ a complete multipartite graph;
\item $\ci$ is the class of the graphs that admit a partition $(C,I)$ of their vertex set such that $C$ induces a cycle, $I$ is an independent set, and every vertex of $I$ has the same neighborhood on $C$.
\end{itemize}

These classes respectively correspond to the outcomes \itemref{h3-3} and \itemref{h3-4} of \autoref{t:dec-h3}.
Our proof of \autoref{t:exclh3-wqo} relies on the two following lemmas which are proved in the next~sections.

\begin{lemma}\label{l:wqo-sub}
  For every (unlabeled) graph $G$ and every wqo $(\Sigma, \lleq),$ the class of
  $(\Sigma, \lleq)$-labeled $G$-subdivisions is well-quasi-ordered by
  contractions.
\end{lemma}

\begin{lemma}\label{t:c-ciwm}
  For every wqo $(\Sigma, \lleq),$ the classes $\lab_{(\Sigma, \lleq)}(\ci)$ and $\lab_{(\Sigma, \lleq)}(\ci)$ are wqo by induced minors.
\end{lemma}

We also use the following result by Thomas.
\begin{proposition}[\cite{Thomas1985240}]\label{p:thomas}
  For every wqo $(\Sigma, \lleq)$, the class of $(\Sigma, \lleq)$-labeled $K_4$-induced minor-free graphs is wqo by induced minors.
\end{proposition}

We now show that $\htg$-induced minor-free graphs are wqo by induced minors.
\begin{proof}[Proof of \autoref{t:exclh3-wqo}]
According to \autoref{p:2c-labels}, it is enough to show that for every wqo
$(\Sigma, \lleq),$ the class of $(\Sigma, \lleq)$-labeled 2-connected graphs not
containing $\htg$ as induced minor is wqo by induced
minors. By~\autoref{t:dec-h3}, this class can be divided into three
subclasses:

\begin{itemize}
\item $K_4$-induced minor-free graphs;  
\item subdivisions of a graph among $K_4$, $K_{3,3}$, and the prism;\label{e:k4-prism}
\item graphs of $\wm \cup  \ci$.
\end{itemize}

\autoref{p:thomas}, \autoref{l:wqo-sub}, and \autoref{t:c-ciwm}
respectively handle these three cases.
Since it is a finite union of wqos, the class of $(\Sigma, \lleq)$-labeled $\htg$-induced minor free graphs is a wqo as well (wrt. induced minors). This
concludes the proof.
\end{proof}

The following sections contain the proofs of
\autoref{l:wqo-sub} and~\autoref{t:c-ciwm}. The technique that
we repeatedly use in order to show that a poset $(A, \lleq_A)$ is a
wqo is the following:
\begin{enumerate}
\item we define a function $f \colon A' \to A.$
  Intuitively, elements of $A'$ can be seen as descriptions (or
  encodings) of objects of $A$ and $f$ is the function constructing
  the objects from the descriptions;
\item we show that $A'$ is wqo by some relation $\lleq_{A'}.$ Usually,
  $A'$ is a product, union or sequence over known wqos so this can be
  done using \autoref{wqomagic};
\item we prove that $f\colon(A', \lleq_{A'}) \to (A, \lleq_A)$ is monotone and surjective
  (sometimes using \autoref{r:mono-pair}) and by \autoref{wqomagic}
  we conclude that $(A, \lleq_A)$ is~a~wqo.
\end{enumerate}

\subsection{Well-quasi-ordering subdivisions}

Let $\opath$ denote the class of paths with at least two vertices and
whose endpoints are 
distinguished, i.e.~one end is said to be the beginning and the other
one the end. In the sequel, $\fst(P)$ denotes the first vertex of the
path $P$ and $\lst(P)$ its last vertex. We extend the relation
$\linm$ to $\opath$ as follows: for every $G,H \in
\opath,\ G \linm H$ if there is an induced minor model $\mu$
of $G$ in $H$ such that $\fst(H) \in \mu(\fst(G))$ and $\lst(H) \in
\mu(\lst(G)),$ and similarly for $\lctr.$

We omit the proof of the following lemma, which follows from the natural correspondence between labeled paths with distinguished ends and sequences of these labels.
\begin{lemma}\label{l:path-wqo}
   For every wqo $(\Sigma, \lleq)$, the poset $(\lab_{(\Sigma,\lleq)}(\opath), \lctr)$ is a wqo.
 \end{lemma}
 
We can now show that labeled subdivisions of a fixed graph are well-quasi-ordered by the contraction relation, i.e.\ \autoref{l:wqo-sub}.

\begin{proof}[Proof of \autoref{l:wqo-sub}]
  Let $G$ be a non labeled graph, let $(\Sigma, \lleq)$ be a wqo and let
  $\mathcal{G}$ be the class of all $(\Sigma, \lleq)$-labeled $G$-subdivisions. We
  set $m = \card{E(G)}.$ Let us show that $(\mathcal{G},
  \lctr)$ is a wqo. First, we arbitrarily choose an orientation to
  every edge of $G$ and an enumeration $e_1, \dots, e_m$ of these
  edges. We now consider the function $f$ that, given a tuple $(Q_1,
  \dots, Q_m)$ of $m$ paths of $\lab_{(\Sigma,\lleq)}(\opath),$ returns the
  graph constructed from $G$ by, for every $i \in \intv{1}{m},$
  replacing the edge $e_i$ by the path $Q_i,$ while respecting the
  orientation, i.e.~the first (resp.\ last) vertex of $Q_i$ goes to
  the first (resp.\ last) vertex of $e_i.$
  As a Cartesian product of wqos
  and since $(\lab_{(\Sigma,\lleq)}(\opath), \lctr)$ is a wqo
  (\autoref{l:path-wqo}), the domain $\lab_{(\Sigma,\lleq)}(\opath)^m$ of $f$
  is well-quasi-ordered by $\lctr^m.$ Notice that
  $f$ is surjective
  on~$\mathcal{G}$.
  In order to show that $(\mathcal{G},\lctr)$ is a wqo, it is enough
  to prove that
  \[f \colon (\lab_{(\Sigma,\lleq)}(\opath), \lctr^m) \to
    (\mathcal{G}, \lctr)\]
  is an monotone, as explained in
  \autoref{wqomagic}, that is, to prove that for every $\mathcal{Q}, \mathcal{R} \in \lab_{(\Sigma,\lleq)}(\opath)^m$ such
  that $\mathcal{Q} \lctr^m \mathcal{R},$ we have $f(\mathcal{Q})
  \lctr f(\mathcal{R}).$ According to \autoref{r:mono-pair}, we
  only need to care of the case where $\mathcal{Q}$ and $\mathcal{R}$
  differ by only one coordinate. By symmetry we may assume that
  they only differ by the first one, i.e. $\mathcal{Q}=(Q_,Q_2, \dots, Q_m)$ and
  $\mathcal{R}=(R, Q_2, \dots, Q_m)$ with $Q
  \lctr R.$ Let $\mu \colon V(Q) \to
  \fpowset(V(R))$ be a contraction model of $Q$ in $R$ and
  let $\mu' \colon V(f(\mathcal{Q})) \to
  \fpowset(V(f(\mathcal{Q})))$ be the trivial contraction
  model of $f(\mathcal{Q}) \setminus V(Q)$ in itself defined by $\forall u \in
  V(f(\mathcal{Q}))\setminus V(Q),\ \mu'(u) = \{u\}.$
  We now consider the function $\nu \colon V(f(\mathcal{Q}))
  \to \fpowset(V(f(\mathcal{R})))$ defined as follows:
  \[
  \nu \colon
  \left \{
    \begin{array}{rcll}
      V(f(\mathcal{Q})) & \to &
      \fpowset(V(f(\mathcal{R})))\\
      u & \mapsto & \mu(u) \quad \text{if}\ u \in V(Q)\\
        u & \mapsto & \mu'(u) \quad \text{otherwise}.
    \end{array}
  \right .
\]

It can be easily checked that $\nu$ is a contraction model of $f(\mathcal{Q})$ in
$f(\mathcal{R}).$ Hence $f(\mathcal{Q}) \lctr f(\mathcal{R})$, as required.
This proves that $(\mathcal{G}, \lctr)$ is a wqo.
\end{proof}

\subsection{Well-quasi-ordering \texorpdfstring{$\wm$ and $\ci$}{WM and CI}}

In this section, we prove \autoref{t:c-ciwm} by first dealing with $\wm$ (\autoref{l:wm-wqo}) and then with $\ci$ (\autoref{l:wqoci}).
The following result is
straightforward consequence of Higman's lemma (see \autoref{wqomagic}).

\begin{corollary}\label{l:indset-wqo}
  If $(\Sigma, \lleq)$ is wqo then the class of $(\Sigma, \lleq)$-labeled independent
  sets (resp.\ cliques) is wqo induced subgraphs.
\end{corollary}

\begin{corollary}\label{c:closure}
  If a class of $(\Sigma, \lleq)$-labeled graphs $(\mathcal{G}, \linm)$ is wqo,
  then so is its closure by finite disjoint union (resp.\ join).
\end{corollary}

\begin{proof}
Let $\mathcal{U}$ be the closure of $\lab_{(\Sigma, \lleq)}(\mathcal{G}, \linm)$ by
disjoint union. Observe that every graph of $\mathcal{U}$ can be
partitioned in a family of pairwise non-adjacent graphs of
$\mathcal{G}.$ Therefore we can define a function mapping every
$\mathcal{G}$-labeled independent set to the graph of $\mathcal{U}$
obtained from $G$ by replacing each vertex by its label (which is an
$(\Sigma, \lleq)$-labeled graph). It is easy to check that this function of $(\mathcal{G}, \linm) \to (\mathcal{U}, \linm)$ is monotone and surjective.
Together with \autoref{wqomagic} and \autoref{l:indset-wqo}, this yields
the desired result.
\end{proof}

\begin{corollary}\label{c:kmp-wqo}
  If $(\Sigma, \lleq)$ is a wqo then the class of 
  $(\sigma, \lleq)$-labeled complete multipartite graphs are wqo by induced
  subgraphs.
\end{corollary}

\begin{lemma}\label{l:wm-wqo}
  If $(\Sigma, \lleq)$ is a wqo then the class of $(\Sigma, \lleq)$-labeled graphs of $\wm$ is well-quasi-ordered by induced subgraphs.
\end{lemma}

\begin{proof}
For every graph $G$ of $\wm$, let $(W_G, M_G)$ be a partition of $V(G)$ as in the definition of $\wm$.
Let $S$ be a wheel on at most 5 vertices and let us consider the subclass $\wm(S)$ of all $(\Sigma, \lleq)$-labeled graphs $G \in \wm$ such that $W_G$ is (isomorphic to) the wheel~$S$. We set $s=|S|$ and choose an ordering $v_1, \dots, v_s$ of the vertices of $S$.
In this proof, for every $G,H \in \wm$, we write $H \lisgrp G$ if $H$ is an induced subgraph of $G$ that can be obtained without deleting vertices of $W_G$.
Observe that well-quasi-ordering by $\lisgrp$ implies well-quasi-ordering by $\lisgr$. Let us show that $\wm(S)$ is well-quasi-ordered by~$\lisgrp$.

For every $G \in \wm(S)$ and every $v \in V(G)$, let $\tau(v) \subseteq \{0,1\}^{s}$ be a tuple encoding the adjacencies of $v$ on $S$. Formally, for a fixed ordering of $V(S)$, the $i$-th coordinate of $\tau(v)$ is equal to 1 if $v$ is adjacent to the $i$-th vertex of $S$ and to 0 otherwise, for every $i\in \intv{1}{s}$.
Recall that for every $G \in \wm(S)$, the label $\lambda_G(v)$ of a vertex $v\in V(G)$ is a finite subset of~$\Sigma$. Let $\lambda_G'(v) = \{(l,\tau(v)),\ l\in \lambda(v)\}$. Informally, we add to the label of $v$ information about its adjacency in $S$. The new label is a subset of $\Sigma\times \{0,1\}^s$.
Let $f$ be the function mapping every $G \in \wm(S)$ to $(\lambda_G(v_1), \dots, \lambda_G(v_s), J_G)$,
where $J_G$ is the graph obtained from $G[M_G]$ by relabeling every vertex $v$ with $\lambda'_G(v)$. We see $J_G$ as a $(\Sigma\times \{0,1\}^s, \lleq \times =)$-labeled graph.
Observe that $f$ is injective and that for every $G,H \in \wm(S)$,
\[
  H \lisgrp G \iff f(H) \mathbin{\underbrace{\lleq^\star \times \dots \times \lleq^\star}_{s\ \text{times}} \times \lisgr} f(G).
\]
According to \autoref{wqomagic}, $(\wm(S), \lisgrp)$ (the domain of $f$) is a wqo iff the image of $f$ is wqo by $\underbrace{\lleq^\star \times \dots \times \lleq^\star}_{s\ \text{times}} \times \lisgr$.
For every $G$, $J_G$ is a complete multipartite graph labeled with a wqo  hence, according to \autoref{c:kmp-wqo}, $\mathcal{J} = \{J_G,\ G \in\wm(S)\}$ is well-quasi-ordered by $\lisgr$.
The image of $f$ is a subset of the Cartesian product of several occurrences of $(\Sigma, \lleq)$ with $(\mathcal{J}, \lisgr)$, thus it is a wqo. This implies that $(\wm(S), \lisgr)$ is a wqo.

Since there is a finite number of non-isomorphic wheels on at most 5 vertices, the $(\Sigma, \lleq)$-labeled graphs of $\wm$ form a finite union of wqo, hence they are wqo by $\lisgr$.
\end{proof}

\begin{lemma}\label{l:wqoci}
  If $(\Sigma, \lleq)$ is a wqo then the class of $(\Sigma, \lleq)$-labeled graphs of $\ci$ is wqo by induced minors.
\end{lemma}

\begin{proof}
  We consider the function
  \[
    f \colon (\lab_{(\Sigma,\lleq)}(\opath)^\star \times (\fpowset(\Sigma), \lleq^{\mathcal{P}})^\star, \lctr^\star \times \lleq^{\mathcal{P}\star}) \to (\lab_{(\Sigma,
    \lleq)}(\ci), \linm)
\] that, given a sequence $\seqb{R_0,\dots,
  R_{k-1}}$ of $(\Sigma, \lleq)$-labeled paths of $\opath$
and a sequence $A$ of subsets of $\Sigma$, returns the graph
  constructed as follows:

   \begin{enumerate}[(i)]
   \item for every element $a \in A$, create a new vertex and label it with~$a$;\label{setp1}
  \item in the disjoint union of these vertices and the paths of $\{R_i\}_{i \in
      \intv{0}{k-1}}$, add an edge between $\lst(R_i)$ and $\fst(R_{(i + 1) \mod
      k})$, for every $i \in \intv{0}{k-1}$;
  \item add all possible edges between $\fst(R_i)$ and the vertices created in the first step, for every $i \in \intv{0}{k-1}$.
  \end{enumerate}

  The domain of $f$ is a wqo, as a Cartesian product of wqos (cf.\ \autoref{l:path-wqo}).
  Observe that its image is~$\lab_{(\Sigma,\lleq)}(\ci)$.
  To show that this set is wqo, it is enough to prove that $f$ is monotone, according to \autoref{wqomagic}. By \autoref{r:mono-pair}, this can be done by proving the two following implications:
  \begin{align*}
    \forall R \in \lab_{(\Sigma,\lleq)}(\opath)^\star,\ \forall A,B \in \Sigma^{\mathcal{P}\star},\ &A \lleq^{\star} B \Rightarrow f(R, A) \linm f(R, B),\ \text{and}\\
    \forall Q,R \in \lab_{(\Sigma,\lleq)}(\opath)^\star,\ \forall A \in \Sigma^{\mathcal{P}\star},\ &Q \lctr^\star R \Rightarrow f(Q,A) \linm f(R,J).    
  \end{align*}

  \noindent \textit{First implication.} Let us call $a_1, \dots, a_{|A|}$ and $b_1, \dots, b_{|B|}$ the elements of $A$ and $B$, respectively.
  As $A \lleq^{\star} B$, there is a function $\varphi \colon \intv{1}{|A|} \to \intv{1}{|B|}$ such that for every $i \in \intv{1}{|A|}$, we have $a_i \lleq b_{\varphi(i)}$.
Let us call $v_i$ the vertex labeled $b_i$ in  step \itemref{setp1} of the construction of $f(R,B)$.
  Then $f(R,A)$ can be obtained from $f(R,B)$ by first deleting the vertices the form $v_i$ with $i \in \intv{1}{|B|} \setminus \varphi(\intv{1}{|A|})$ and then, for every $i \in \intv{1}{|A|}$, contracting the label of the vertex $v_{\varphi(i)}$ (which is $b_{\varphi(i)}$) to~$a_i$. Hence $f(R, A) \linm f(R, B)$.

  \smallskip
  \noindent \textit{Second implication.}
Let $Q_0, \dots, Q_{k-1}$ and $R_0, \dots,
R_{l-1}$ be the elements of $Q$ and $R$, respectively.
By definition of the relation $\lctr^\star,$ there is an increasing
function $\varphi \colon \intv{0}{k-1} \to \intv{0}{l-1}$ such that
\[
\forall i \in \intv{0}{k-1},\ Q_i \lctr R_{\varphi(i)}.
\]
Let us call $\mu_i$ a contraction model of $Q_i$ in $R_i$, for every $i \in \intv{0}{k-1}$.
Let $\mu \colon V(f(Q,A)) \to \fpowset(V(f(R, A)))$ be the function that maps a vertex $v$
\begin{itemize}
\item to $\{v\}$ if it has been created during step \itemref{setp1} of the construction of $f(Q,A)$;
\item to $\mu_i(v)$ otherwise, where $i$ is such that $v \in Q_i$.
\end{itemize}
It can be checked that $\mu$ is a contraction model of $f(Q,A)$ in $f(R,A)$. As a consequence, $f(Q,A) \linm f(R,A)$.

\smallskip
We deduce that $\lab_{(\Sigma, \lleq)}(\ci)$ is wqo by induced minors, as desired.
\end{proof}

\section{Graphs not containing
  \texorpdfstring{$\gem$}{Gem}}
\label{sec:wqogem}

The purpose of this section to give a proof to
\autoref{t:exclgem-wqo}. This will be done by first proving a
decomposition theorem for graphs of $\excl(\gem)$, and then using this
theorem to prove that $(\excl(\gem), \linm)$ is a wqo.

\subsection{A Decomposition theorem for \texorpdfstring{$\excl(\gem)$}{Excl(Gem)}}

This section is devoted to the proof of \autoref{t:decgem}, which
is split in several lemmas. 
In the sequel, $G$ is a 2-connected graph of~$\excl(\gem)$.
When $G$ is 3-connected, we will rely on the following result
originally proved by Ponomarenko.
\begin{proposition}[\cite{ponomarenko91}]
	\label{l:pono}
	Every 3-connected $\gem$-induced minor-free graph is either a
        cograph, or has an induced subgraph $S$ isomorphic to $P_4$,
        such that every connected component of $G \setminus S$ is a cograph.
\end{proposition}
Therefore we will here focus on the case where
$G$ is 2-connected but not 3-connected.
A \emph{rooted diamond} is a graph which can be constructed from a rooted $C_4$ by adding a chord incident with exactly one endpoint of the root (cf.~\autoref{fig:rd}).

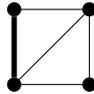
\begin{figure}[ht]
	\centering
	\begin{tikzpicture}[every node/.style = black node]
          \draw (0,1) node (C) {} -- (1,1) node (B) {} -- (1,0) node (D) {} -- (0,0) node (A) {} -- (B);
          \draw[line width = 2] (A) -- (C);
	\end{tikzpicture}
	\caption{A rooted diamond, the root being the thick edge.\label{fig:rd}}
\end{figure}

\begin{lemma}\label{l:rdtogem}
  Let $S=\{v_1, v_2\}\in E(G)$ be a cutset in a graph $G$ and let $J$ be a component of $G\setminus S$. Let $H$ be the graph
  $\induced{G}{V(J) \cup \{v_1, v_2\}}$ rooted at
  $\{v_1,v_2\}$. If $H$ has a rooted diamond as induced minor, then~$\gem \linm
  G$.
\end{lemma}
\begin{proof}
  Let $J'$ be a component of $G\setminus S$ other than $J$ and let
  $G'$ be the graph obtained from $G$ by:
  \begin{enumerate}
  \item applying the necessary operations (contractions and vertex
    deletions) to transform $\induced{G}{V(J) \cup \{v_1,
      v_2\}}$ into a rooted diamond; 
  \item deleting every vertex not belonging to $V(J)
    \cup V(J') \cup \{v_1, v_2\}$;
  \item contracting $J'$ to a single vertex.
  \end{enumerate}
  The graph $G'$ is then a rooted diamond and a vertex adjacent to
  both endpoints of its root, that is, $G'$ is isomorphic to $\gem$.
\end{proof}

Let us now characterize these 2-connected graphs avoiding rooted diamonds.
\begin{lemma}\label{l:diam-free}
  Let $G$ be a 2-connected graph rooted at $\{u,v\} \in E(G)$. If $\{u,v\}$
  is not a cutset of $G$ and $G$ does not contain a rooted diamond as
  induced minor, then either $G$ is an induced cycle, or both $u$ and $v$ are dominating in $G$.
\end{lemma}

\begin{proof}
  Assuming that $u$ is not dominating and $G$ is not an induced
  cycle, let us prove that $G$ contains a rooted diamond as
  induced minor.
  Let $w \in V(G)$ be a vertex such that $\{u, w\} \not
  \in E(G)$. Such a vertex always exists given that $u$ is
  not dominating. Let~$C$ be a shortest cycle using the edge
  $\{u,v\}$ and the vertex $w$ (which exists since $G$ is
  2-connected), let $P_u$ be the subpath of $C$ linking $u$ to
  $w$ without meeting $v$ and similarly let $P_v$ be the subpath
  of $C$ linking $v$ to $w$ without meeting $u$.
  By the choice of $C$, both $P_u$ and $P_v$ are induced
  paths. Notice that if there is an edge other than $\{u,v\}$ connecting a vertex of
  $P_u \setminus \{w\}$ to vertex of $P_v \setminus \{w\}$, then
  $G$ contains a rooted diamond as induced minor. Therefore we can now assume
  that $C$ is an induced~cycle.
  
  Since we initially assumed that $G$ is not an induced
  cycle, $G$ contains a vertex not belonging to $C$.
  Let $G'$ be the graph obtained from $G$ by contracting to one
  vertex $x$ some connected component of $G \setminus C$ and deleting all the
  other components. Obviously we have~$G' \linm G$. Let us show
  that $G'$ contains a rooted diamond as induced minor.
  
  Observe that $G'$ is a $k$-wheel of center $x$, for some $k\geq 2$. Furthermore, the neighborhood of $x$, which has size at least
  two (as $G$ is 2-connected), is not equal to $\{u,v\}$,
  otherwise $\{u,v\}$ would be a cutset in~$G$.
  It is then not hard to see that $G'$ contains a rooted diamond as induced minor, a contradiction.
\end{proof}
      
\begin{remark}
  \label{r:domcogr}
        In a $\gem$-induced minor-free graph $G$, every induced
        subgraph $H$ dominated by a vertex $v \in V(G)
        \setminus V(H)$ is a cograph.
\end{remark}
Indeed, assuming that $H$ is not a cograph, let $P$ be a path on four
vertices which is induced subgraph of $H$. Then $\induced{G}{V(P) \cup
\{v\}}$ is isomorphic to $\gem$, a contradiction.

Recall that we say that an induced subgraph of $G$ is \emph{basic in $G$}
if it is either a cograph, or an induced path whose internal vertices
are of degree two in~$G$.
\begin{lemma} \label{l:k2-cut}
  If $G$ has a $K_{2}$-cutset $S = \{ v_1, v_2 \}$, then every connected
  component of $G \setminus S$ is basic in $G$.
\end{lemma}
\begin{proof}
  By \autoref{l:rdtogem}, for every connected component $J$ of
  $G\setminus S$ we know that the graph $\induced{G}{V(J) \cup
    S}$ rooted at $\{u,v\}$ contains no rooted diamond. By the virtue
  of \autoref{l:diam-free}, this graph either is an induced cycle,
  or has a dominating vertex among $u$ and $v$. 
        In the first case, $J$ is a path whose all internal vertices
        are of degree two in $G$, hence $H$ is basic. If one of $u$
        and $v$ is dominating, then $J$ is a cograph according to
        \autoref{r:domcogr}. Therefore in both cases $C$ is basic
        in~$G$.
\end{proof}

Let us now focus on 2-connected graphs with a $\overline{K_{2}}$-cutset, which is
the last case in our characterization theorem.  

\begin{corollary}
	\label{c:2k1cutcc}
	If $G$ has a $\overline{K_{2}}$-cutset $S$ such that $G \setminus
        S$ contains more than two connected components, then every
        connected component of $G \setminus S$ is basic in~$G$. 
\end{corollary}
\begin{proof}
	It follows directly from \autoref{l:k2-cut}. Indeed, if the
        connected components of $G \setminus S$ are $J_1, J_2, \ldots
        J_k$, let us contract $J_1$ to an edge between the two
        vertices of $S$. The obtained graph fulfills the assumptions
        of \autoref{l:k2-cut}: $S$ is a $K_2$-cutset. Therefore each of
        the components $J_2, \dots, J_k$ is basic in $G$. Applying the
        same argument with $J_2$ instead of $J_1$ yields that $J_1$ is
        basic in $G$ as well.
\end{proof}

\begin{lemma}\label{l:mincycl}
	Let $S = \{u,v\}$ be a $\overline{K_{2}}$-cutset, such that and
        $G \setminus S$ has only two connected components $J_1$ and
        $J_2$. Then $G$ contains a cycle $C$ as induced subgraph such
        that every connected component of $G \setminus C$ is basic in~$G$.
\end{lemma}
\begin{proof}
  For every $i \in \{1,2\}$, let $Q_i$ be a shortest path linking $u$
  to $v$ in $\induced{G}{V(J_i) \cup \{u,v\}}$. Notice that
  the cycle $C = \induced{G}{V(Q_1) \cup V(Q_2)}$ is
  then an induced cycle.
  For contradiction, let us assume that some connected component $J$ of
  $\induced{G}{V \setminus C}$ is not basic in~$G$. By symmetry, we can
  assume that $J \subset J_1$. 
 
  Notice that since $G$ is 2-connected, $J$ has at least two
  distinct neighbors $x,y$ on~$C$.
  Let $G'$ be the graph obtained from $G$ by contracting $Q_1$ to an
  edge between $u$ and $v$ in a way such that $x$ is not contracted to
  $y$ (that is, $x$ is contracted to one of $u,v$ and $y$ to the other
  one). In $G'$, $\{u,v\}$ is a $K_2$-cutset, therefore by \autoref{l:k2-cut}, every connected component of $G\setminus S$ is
  basic in $G'$. As this consequence holds for every choice of $J$ and
  $G'$ is an induced minor of $G$, we eventually get that every
  connected component of $G \setminus C$ is basic in~$G$.
\end{proof}

In the sequel, $S, u,v, C, H_1$, and $H_2$ follow the definitions of the
statement of \autoref{l:mincycl}.

\begin{remark}
	\label{r:at-most-3-neighbors}
	Every connected component $J$ of $G \setminus C$ has at least two and at most three neighbors on $C$.
\end{remark}
Indeed, it has at least two neighbors on $C$ because $G$ is
2-connected. Besides if $J$ has at least four neighbors on $C$, then 
contracting in $\induced{G}{V(C) \cup V(J)}$ the
component $J$ to a single vertex, deleting a vertex of $C$ not
belonging to $N(J)$ (which exists since $J$ belongs to only one
of the components of $G \setminus S$) and then contracting every edge
incident with a vertex of degree two would yield~$\gem$.

\begin{lemma}
	\label{l:neighborhood-is-equal}
        If $C$ has at least one vertex of degree two, then for every distinct connected components $J_1$ and $J_2$ of $G \setminus C$ we have $N_C(J_1) \subseteq N_C(J_2)$ or $N_C(J_2) \subseteq N_C(J_1)$.
\end{lemma}
\begin{proof}
  Let us assume, for contradiction, that the claim is not true and let $G$ be
  a minimal counterexample with respect to induced minors. In such a
  case both $J_1$ and $J_2$ are single vertices (say $j_1$ and $j_2$
  respectively) and they are the only connected components of~$G
  \setminus C$. We now argue that any such minimal
  counterexample must contain as induced minor one of graphs presented on
  \autoref{fig:similar-neighborhood} (where thick edges represent
  the cycle~$C$). This would conclude the
  proof as each of these graphs contains $\gem$ as induced minor, as
  shown in \autoref{fig:similar-neighborhood}.

  First of all, in such a minimal counterexample there is only one
  vertex in $C$ of degree $2$, let us call it~$c$. We will consider all
  the ways that the vertices $j_1$ and $j_2$ can be connected to the
  neighbors of $c$, and show that in every such case we can contract
  our graph to one of the graphs on
  \autoref{fig:similar-neighborhood}. According to
  \autoref{r:at-most-3-neighbors}, each of $j_1$ and $j_2$ will have
  either two or three neighbors on $C$.

    \smallskip
  \noindent \textit{First case:} both $j_1$ and $j_2$ are connected with
    both neighbors of $c$. As $N(j_1) \not \subseteq
  N(j_2)$ and $N(j_2) \not \subseteq N(j_1)$, each of
  $j_1,j_2$ has a neighbor which is not adjacent to the other. But
  since $j_1$ and $j_2$ can have at most three neighbors, the
  neighborhood of $j_1$ and $j_2$ is now completely characterized.
  \autoref{fig:similar-neighborhood}.(a) presents
  the only possible graph for this case.

  \smallskip
  \noindent \textit{Second case:} $j_1$ is connected with exactly one of
    neighbors of $c$ and $j_2$ is connected with the other one. In
  this case, as each of $j_1,j_2$ has at least two neighbors on $C$,
  contracting all the edges of $C$ whose both endpoints are at distance at
  least two from $c$ gives the graph depicted in
  \autoref{fig:similar-neighborhood}.(b).

  \smallskip
  \noindent \textit{Third case:} $j_1$ is connected with both neighbors of
    $c$, and $j_2$ is connected with at most one of them.
  In this case, as long as $C$ has more than $4$ edges, we can contract an edge of $C$
  to find a smaller counterexample. Precisely, if it has more than $4$ edges, there
  are two edges $e_1, e_2$ in $C$ within distance exactly one to $c$ and those
  two do not share an endpoint. Moreover $j_2$ has a neighbor $s$ in 
  $C \setminus N(c)$, which is not a neighbor of $j_1$. Now
  one of the edges $e_1, e_2$ is not incident to $s$, and contracting this edge yields
  a smaller counterexample. Therefore, we only have to care about the case where $C$ has exactly $4$ edges,
  and this case is exactly the graph represented on \autoref{fig:similar-neighborhood}.(c).

  In each of the induced minor-minimal counterexamples, a $\gem$ can be found as induced minor, as depicted in \autoref{fig:similar-neighborhood}.
  This concludes the proof.
\end{proof}
\begin{figure}[ht]
\centering
	\begin{tikzpicture}[every node/.style = black node, scale = 1.25]
          \begin{scope}
            \begin{scope}[color = blue!25, every node/.style = {normal, minimum size = 0, inner sep=3pt}]
              \fill (-1,0) circle (0.15cm);
              \fill (0,1) circle (0.15cm);
              \fill (0,-0.333) circle (0.15cm);
              \fill (1,-0.5) circle (0.15cm);
              \fill (0, 0.333) circle (0.15cm);
              \fill (1,0.5) circle (0.15cm);
              \fill (0,-1) circle (0.15cm);
              \draw[color = blue!25, line width =0.25cm] (0, 0.333) -- (1,0.5) (1,0.5) -- (0,-1);
            \end{scope}
			\node[white node] (c) at (-1,0) [label=left:2] {};
			\node[label=90:3] (1) at (0, 1) {};
			\node (2) at (0, 0.333) {};
			\node[label=135:5] (3) at (0,-0.333) {};
			\node[label = -90:1] (4) at (0, -1) {};
			\node (x1) at (1, 0.5) [label=right:$j_1$] {};
			\node[label=right:$j_2$, label=-90:4] (x2) at (1, -0.5) {};
			\draw [very thick] (c) -- (1) -- (2) -- (3) -- (4) -- (c);
			\draw (1) -- (x1) -- (2);
			\draw (x1) -- (4);
			\draw (3) -- (x2);
			\draw (x2) -- (4);
			\draw (x2) -- (1);
                        \draw (0, -1.75) node[normal] {(a)};
		\end{scope}
		\begin{scope}[xshift=100pt]
            \begin{scope}[color = blue!25, every node/.style = {normal, minimum size = 0, inner sep=3pt}]
              \fill (-1,0) circle (0.15cm);
              \fill (0,1) circle (0.15cm);
              \fill (0,0) circle (0.15cm);
              \fill (1,-0.5) circle (0.15cm);
              \fill (1,0.5) circle (0.15cm);
              \fill (0,-1) circle (0.15cm);
              \draw[color = blue!25, line width =0.25cm] (-1,0) -- (0,1);
            \end{scope}
			\node[white node] (c) at (-1,0) {};
			\node[label=90:3] (1) at (0, 1) {};
			\node[label=180:1] (2) at (0, 0) {};
			\node[label=-90:4] (3) at (0, -1) {};
			\node[label=90:2] (x1) at (1, 0.5) [label=right:$j_1$] {};
			\node[label=-90:5] (x2) at (1, -0.5) [label=right:$j_2$] {};

			\draw [very thick] (c) -- (1) -- (2) -- (3) -- (c);
			\draw (1) -- (x1) -- (2);
			\draw (2) -- (x2) -- (3);
                        \draw (0, -1.75) node[normal] {(b)};
		\end{scope}
		\begin{scope}[xshift=200pt]
                  \begin{scope}[color = blue!25, every node/.style = {normal, minimum size = 0, inner sep=3pt}]
                    \fill (-1,0) circle (0.15cm);
                    \fill (0,1) circle (0.15cm);
                    \fill (0,0) circle (0.15cm);
                    \fill (1,-0.5) circle (0.15cm);
                    \fill (1,0.5) circle (0.15cm);
                    \fill (0,-1) circle (0.15cm);
                    \draw[color = blue!25, line width =0.25cm] (-1,0) -- (0,1);
                  \end{scope}
			\node[white node] (c) at (-1,0) {};
			\node[label=90:1] (1) at (0, 1) {};
			\node[label=180:4] (2) at (0, 0) {};
			\node[label=-90:3] (3) at (0, -1) {};
			\node[label=-90:5] (x1) at (1, -0.5) [label=right:$j_2$] {};
			\node[label=90:2] (x2) at (1, 0.5) [label=right:$j_1$] {};

			\draw [very thick] (c) -- (1) -- (2) -- (3) -- (c);
			\draw (3) -- (x1) -- (2);
			\draw (3) -- (x2);
			\draw (x2) -- (1);
                        \draw (0, -1.75) node[normal] {(c)};
		\end{scope}
	\end{tikzpicture}
	\caption{Induced minor-minimal counterexamples in the proof \autoref{l:neighborhood-is-equal} contain the $\gem$ as induced minor. The vertex $c$ is depicted in white. The numbers indicate which vertices of the $\gem$ (following the convention of \autoref{fig:h123}) correspond to the subsets of vertices depicted in blue.}
	\label{fig:similar-neighborhood}
      \end{figure}
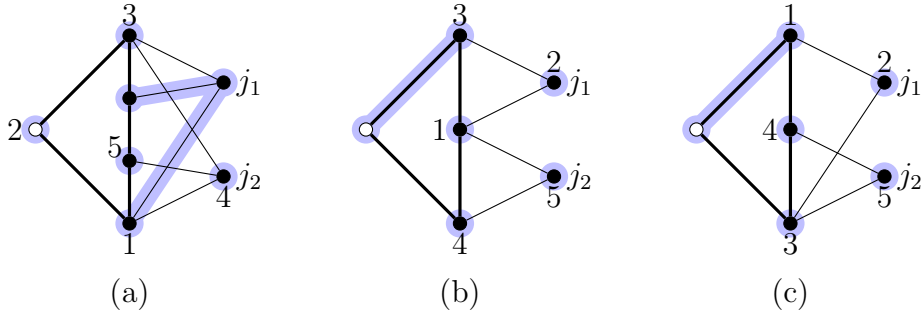

\begin{corollary}\label{c:dggt}
	If $C$ has at least one vertex of degree two, then it has at most three vertices of degree greater than two.
\end{corollary}
\begin{proof}
  Notice that the set of vertices of $C$ that have degree greater then two is
  exactly the union of $N_C(J)$ over all connected components $J$
  of $G \setminus C$. We just saw in
  \autoref{l:neighborhood-is-equal} that for every two connected
  components of $G \setminus C$, the neighborhood on
  $C$ of one is contained in the neighborhood on $C$ of the other. Besides
  these neighborhoods have size at most three, otherwise we would be able to find a $\gem$ as induced minor. Therefore their union
  have size at most three as~well.
\end{proof}

\begin{corollary}
	\label{k2bar-finalcut}
	Every connected component of $G\setminus C$ is basic and $C$
        has at most six vertices of degree greater than two.
\end{corollary}
\begin{proof}
 Notice that contracting $J_1$ to a single vertex $h$ in $G$ gives a
 graph $G'$ and a cycle $C'$ (contraction of $C$) such that every
 connected component of $G'\setminus C'$ is basic and $C'$ has at
 least one vertex of degree 2,~$h$. By \autoref{c:dggt}, $C'$
 has at most three vertices of degree greater than two. Notice that
 these vertices belong to $G' \setminus h$ which is isomorphic to $G
 \setminus J_1$. Hence $G \setminus J_1$ has at most three vertices of
 degree greater than two. Applying the same argument with $J_2$
 instead of $J_1$ we get the desired~result.
\end{proof}

Now we are ready to prove the main decomposition theorem for $\gem$-induced minor-free graphs.

\begin{proof}[Proof of \autoref{t:decgem}]
  Recall that we are looking for a subset $X$ of $V(G)$ of
  size at most 6 such that each component of $G \setminus X$ is basic
  in~$G$.

  If $G$ is 3-connected, by \autoref{l:pono} it is either a
  cograph, or has a subset $X$ of four vertices such that every
  connected component of $G \setminus X$ is a cograph. Let us now
  assume that $G$ is not 3-connected.

  In the case where $G$ has a $K_2$-cutset $S$, or if $G$ has a
  $\overline{K_2}$-cutset $S$ such that $G \setminus S$ has more than two connected
  components, then according to \autoref{l:k2-cut} and
  \autoref{c:2k1cutcc} respectively, $S$ satisfies the required properties.
  In the remaining case, by \autoref{k2bar-finalcut} $G$ has a
  cycle $C$ such that every connected component of $G \setminus C$ is
  basic in $G$ and which has at most six vertices of degree more than
  two in~$G$. Let $X$ be the set containing those vertices of degree more
  than two. Observe that
  every connected component of $G \setminus X$ is either a connected
  component of $G \setminus C$ (hence it is basic) or a part of $C$, i.e.\ 
  a path whose internal vertices are of degree two in $G$ (which is basic as
  well). As $|X| \leq 6$, $X$ satisfies the desired properties.
\end{proof}

\subsection{Well-quasi-ordering \texorpdfstring{$\gem$}{Gem}-induced minor-free graphs}

In this section we give a proof of \autoref{t:exclgem-wqo}.
We proved in the previous section that the structure of 2-connected $\gem$-induced minor-free graphs is essentially very simple, with building blocks being cographs and long induced paths. To conclude that labeled 2-connected $\gem$-induced minor-free graphs are wqo by induced minor relation, we need the fact that the building blocks, in particular labeled cographs, are themselves well-quasi-ordered by the induced minor relation. For this we rely on the following extension of the results of Damaschke~\cite{JGT:JGT3190140406} to labeled graphs due to Atminas and Lozin.

\begin{proposition}[from \cite{Atminas14b}]
	\label{t:cographs-wqo}
	For any wqo $(\Sigma, \lleq)$, the class of $(\Sigma, \lleq)$-labeled cographs is wqo by induced subgraphs.
\end{proposition}

Fellows et al.\ proved that if a (labelled) graph class $\mathcal{G}$ is wqo by subgraphs, then for every $k \in \N$, the class of graphs that have $k$ vertices whose deletion results in a graph of $\mathcal{G}$ is also wqo by subgraphs~\cite[Theorem~4]{fellows2009well}. We here prove a counterpart of this result for labeled induced minors.
For every graph class $\mathcal{G}$ ad every integer $k$, we denote by $\mathcal{G}^{(+k)}$ the class of graphs that have at most $k$ vertices whose deletion results in a graph of $\mathcal{G}$.

\begin{lemma}\label{l:plusk}
Let $\mathcal{G}$ be a class of graphs such that for every wqo $(\Sigma, \lleq)$, the class $\lab_{(\Sigma, \lleq)}(\mathcal{G})$ is wqo by induced minors. Then for every $k\in \N$, and every wqo $(\Sigma', \lleq')$, the class $\lab_{(\Sigma', \lleq')}(\mathcal{G}^{(+k)})$ is wqo by induced minors.
\end{lemma}

\begin{proof}
  Graphs of $\mathcal{G}^{+k}$ can be partitioned into $k+1$ classes depending on the minimum number of vertices to delete in order to obtain a graph of $\mathcal{G}$.
  In each of these classes the partition can be refined depending on the subgraph induced by the vertices to remove (for an arbitrarly choice of these vertices).
  Since a finite union of wqos is a wqo, it is enough to focus on the class $\mathcal{H}$ of graphs that have a set $X$ of exactly $k$ vertices such that:
  \begin{itemize}
  \item $G\setminus X \in \mathcal{G}$;
  \item after forgeting the labels, $G[X]$ is (isomorphic to) the same unlabeled graph $H$.
  \end{itemize}
  We fix an ordering $h_1, \dots, h_k$ of the vertices of $H$.
  Let $\Sigma = \Sigma' \times \fpowset(\intv{1}{k})$ and let $\lleq$ be the order $\lleq'  \times =$ on $\Sigma$.
  For every $\lambda \in \fpowset(\Sigma)$, we define $\pi(\lambda)$ as the union of the sets $A \subseteq \intv{1}{k}$ such that $(s, A) \in \lambda$ for some $ s\in \Sigma'$.
  We also set $\tau(\lambda) = \{s,\ \exists A\subseteq \intv{1}{k},\ (s, A) \in \lambda\}$.
  
  Informally, the label of a vertex will encode some adjacencies together with a label and the functions $\pi$ and $\tau$ can be used to retrieve this information.
  Let $f$ be the function that, given a $k$-uple $(\lambda_1, \dots, \lambda_k) \in \fpowset(\Sigma')^k$ and a graph $G \in \lab_{(\Sigma, \lleq)}(\mathcal{G})$, constructs the graph $f(G)$ from the disjoin union of $H$ and $G$ as follows:
  \begin{itemize}
  \item label $h_i$ with $\lambda_i$, for every $i\in \intv{1}{k}$;
  \item make every vertex $v\in V(G)$ adjacent to the vertices of $\{h_1, \dots, h_k\}$ whose indices are given by $\pi(\lambda_G(v))$;
  \item relabel every vertex $v\in V(G)$ with the label given by $\tau(\lambda_G(v))$.
  \end{itemize}

  One can easily check that $\mathcal{H}$ is included in the image of~$f$. Besides, the domain of $f$ is a wqo, as it is a Cartesian product of wqos. In order to show that $\mathcal{H}$ is wqo by $\linm$, we can prove that $f$ is monotone, according to \autoref{wqomagic}.
  As usual we focus on proving two implications:
  \begin{align*}
    \forall \lambda,\lambda', \lambda_2, \dots, \lambda_k \in \fpowset(\Sigma')&, \forall G \in \fpowset(\Sigma'),\\
    &\lambda \lleq'^\star \lambda' \Rightarrow f(\lambda, \lambda_2, \dots, \lambda_k, G) \linm f(\lambda', \lambda_2, \dots, \lambda_k, G),\ \text{and}\\
    \forall \lambda_1,\lambda_2, \dots, \lambda_k \in \fpowset(\Sigma')&, \forall G,G' \in \fpowset(\Sigma'),\\
    &G \linm G' \Rightarrow f(\lambda_1, \lambda_2, \dots, \lambda_k, G) \linm f(\lambda_1, \lambda_2, \dots, \lambda_k, G').
  \end{align*}
We do not consider the cases where the other parameters of $f$ are different since these situations are symmetric to that addressed with the first implication.

\smallskip
  \noindent \textit{First implication.}
  We can obtain $f(\lambda, \lambda_2, \dots, \lambda_k, G)$ from $f(\lambda', \lambda_2, \dots, \lambda_k, G)$ by contracting to $\lambda$ the label $\lambda'$ carried by $h_1$.
  Hence $f(\lambda, \lambda_2, \dots, \lambda_k, G) \linm f(\lambda', \lambda_2, \dots, \lambda_k, G)$.

  \smallskip
  \noindent \textit{Second implication.} For the sake of clarity we set $R = f(\lambda_1, \lambda_2, \dots, \lambda_k, G)$ and $R' = f(\lambda', \lambda_2, \dots, \lambda_k, G)$.
  Let $\mu$ be an induced-minor model of $G$ in $G'$. Let $\mu' \colon V(R) \to  \fpowset(V(R'))$ be the function which maps a vertex $v \in V(R)$ to $\{v\}$ if $v \in V(H)$ and $\mu(v)$ otherwise.
  
  Let us show that $\mu'$ is an induced minor model of $R$ in $R'$.
  The fact that $\mu'$ is a containment model either follow from the properties of $\mu$ (for vertices of $V(G)$) or is trivial (for vertices of $V(H)$).
  To show that it is an induced minor model, we have to prove that for every pair of adjacent vertices $x,y$ of $R$, the sets $\mu'(x)$ and $\mu'(y)$ are connected by an edge in $R'$. Again this is straightforward when $x,y \in V(H)$ or $x,y \in V(G)$, hence we assume that $x \in V(H)$ and $y \in V(G)$. Without loss of generality we may assume that $x = h_1$.
  By construction and since $h_1$ are $y$ adjacent in $R$, the label $\lambda_G(y)$ of $y$ in $G$ contains a pair $(s,A)$ for some $s \in \Sigma'$ and some $A \in \fpowset(\intv{1}{k})$ that contains 1.
  Besides, by definition of $\mu$, we have $\lambda_G(y) \lleq^\star \bigcup_{z \in \mu(y)}\lambda_{G'}(z)$. Hence there is a vertex $z \in \mu(y)$ sucht that $\lambda_{G'}(z)$ contains a pair $(s',A')$ 
  for some $s' \in \Sigma'$ and some $A' \in \fpowset(\intv{1}{k})$ that contains 1. Therefore there is an edge between a vertex of $\mu'(y)$ and one of $\mu'(x)$, as desired. This proves that $\mu'$ is an  induced minor model of $R$ in $R'$. Hence $R \linm R'$, as required.
\end{proof}

\begin{proof}[Proof of \autoref{t:exclgem-wqo}.]
  According to~\autoref{p:2c-labels}, it is enough to prove
  that for every wqo $(S, \lleq)$, the class of $(S, \lleq)$-labeled
  2-connected graphs which does not contain $\gem$ as induced minor
  is well-quasi-ordered by induced minors. By \autoref{t:decgem},
  these graphs can be turned into a disjoint union of paths and
  cographs by the deletion of at most six vertices.
  As a consequence of \autoref{l:plusk} (for $k=6$ and where $\mathcal{G}$ is the class of disjoint unions of cographs and paths), these graphs are well-quasi-ordered by induced minors and we are done.
\end{proof}

\section{Concluding remarks}
\label{sec:fin}

In this paper we characterized all
graphs $H$ such that the class of $H$-induced minor-free graphs is a
well-quasi-order with respect to the induced minor relation. This
allowed us to identify the \textsl{boundary graphs} ($\gem$ and
$\htg$) and to give a dichotomy theorem for this problem.
Our proof relies on two decomposition theorems and a study of infinite
antichains of the induced minor relation. This work can be seen
as the induced minor counterpart of previous dichotomy theorems by
Damaschke~\cite{JGT:JGT3190140406} and
Ding~\cite{Ding:1992:SW:152782.152791}.

The question of characterizing ideals which are
well-quasi-ordered can also be asked for ideals defined by forbidding
several elements. To the knowledge of the authors, very little~\cite{lewchalermvongs2015well}
is known on these classes for the induced minor relation,
and thus their investigation could be the next target in the study of induced minors
ideals. Partial results have been obtained when considering the
induced subgraph relation~\cite{Korpelainen20111813, Lozin13}.


\end{document}